\documentclass[a4paper,12pt]{article}
\usepackage{amsfonts,amssymb,amsthm,amsmath,latexsym,bm,enumerate}

\usepackage{subfigure,epsfig,color,caption, authblk, graphicx, enumitem, mathrsfs, indentfirst}

\usepackage{hyperref, cleveref}
\hypersetup{colorlinks=true, linkcolor=blue, filecolor=magenta, urlcolor=cyan}
\numberwithin{equation}{section}

\newcommand{\keywords}[1]{\small\textbf{\textit{Keywords---}}#1}

\newtheorem{theorem}{Theorem}
\newtheorem{lemma}{Lemma}[section]

\newtheorem{corollary}{Corollary}[section]
\newtheorem{example}{Example}[section]

\title{On the convergence of PINNs for inverse source problem in the complex Ginzburg-Landau equation}

\author[1]{Xing Cheng}
\author[2]{Zhiyuan Li\footnote{Corresponding to: lizhiyuan@nbu.edu.cn}}
\author[3]{Mengmeng Zhang}
\author[4]{Xuezhao Zhang}

\affil[1]{School of Mathematics, Hohai University, Nanjing 210098, China}
\affil[2,4]{School of Mathematics and Statistics, Ningbo University, Ningbo 315211, China}
\affil[3]{School of Science, Hebei University of Technology, Tianjin 300401, China}

\begin{document}
\maketitle

\abstract{ This paper addresses the problem of recovering the spatial profile of the source in the complex Ginzburg-Landau equation from regional observation data at fixed times. We establish  two types of sufficient measurements for the unique solvability of the inverse problem. The first is to determine the source term by using whole data at one fixed instant. Conditional stability is established by using the eigenfunction expansion argument. Next, using the analytic continuation method, both uniqueness and a stability estimate for recovering the unknown source can be established from local data at two instants. Finally, algorithms based on the physics-informed neural networks (PINNs) are proposed, and several numerical experiments are presented to show the accuracy and efficiency of the algorithm.}



\section{Introduction and main results}

The complex Ginzburg-Landau equation (CGLE), originally proposed by Ginzburg and Landau in \cite{ginzburg1950}, is one of the fundamental theoretical models used to capture the physical processes of superconducting states initially. 
With the development of theoretical research on CGLE, it is now used much more widely. Now it plays a central role in various scientific fields, including phase transitions in non-equilibrium systems, instabilities in fluid dynamics, chemical turbulence and reaction-diffusion systems, and nonlinear optics \cite{hohenberg2015, rosenstein2010}. In particular, the CGLE has emerged as a canonical model for dissipative nonlinear systems, providing crucial insights into complex phenomena such as non-equilibrium dynamics, self-organization, and spatiotemporal chaos \cite{aranson2002,  cheng2023, Liu2017, garciam2012}.

In recent years, the CGLE has attracted great attention from mathematicians. Substantial progress has been made, including well-posedness, uniqueness, and regularity of solutions. 
For instance, 
Shimotsuma, Yokota, and  Yoshii \cite{shimotsuma2016} provided decay estimates and analyzed the local and global solvability of the CGLE under various spatial and parameter conditions. Correa and \"Ozsari \cite{correa2018} formulated a well-posedness theory with dynamic boundary conditions, encompassing both strong and weak solutions. For the CGLE with Robin-type boundary condition, we refer to Kaikina \cite{kaikina2019} in which existence and regularity of the solution was further addressed. Recently, some mathematicians have also begun to study stochastic CGLE, see e.g., Balanzario and Kaikina \cite{balanzario2020} investigated a stochastic nonlinear CGLE on a half-line with white noise boundary conditions, proving local well-posedness and analyzing regularity near the boundary. We also refer to the recent progress made by Cheng, Guo and Zheng \cite{CGZ0,cheng2023} on the CGLE in the whole space.

The inverse problems associated with the complex Ginzburg–Landau equation have also been widely studied, although we do not aim to give a complete bibliography here. Dou, Fu, Liao, and Zhu \cite{dou2023} derived global Carleman estimates for the CGLE with cubic nonlinearity, and applied them to state observation problems. Fan and Jiang \cite{fan2005} proved existence and uniqueness results for an inverse problem in a time-dependent CGLE model with integral type overdetermination. In a further study, Fan, Jiang, and Nakamura \cite{fan2010} examined existence, uniqueness, and stability for the inverse problem with final time data. 
Rosier and Zhang \cite{rosier2008} analyzed internal and boundary control problems by combining Carleman estimates and sectorial operator theory, obtaining null controllability results. Significant progress has also been made in the numerical analysis of inverse source problems related to the CGLE. Trong, Duy, and  Minh \cite{trong2016} proposed a function approximation and regularization approach to solve the ill-posed terminal value problem with locally Lipschitz nonlinearities, deriving error bounds under noisy data. Kirane, Nane, and Tuan \cite{kirane2017} introduced a novel regularization technique incorporating statistical methods and proved convergence rates in both $L^2$ and $H^1$ norms. 

As a limiting case of the CGLE, their analytical frameworks and inverse problem methodologies share intrinsic connections. Notably, techniques such as Carleman estimates \cite{Yamamoto2009,Bellassoued2010,Wu2020,Chen2022,Zheng2014} can be adaptively modified for CGLE inverse problems by incorporating complex-valued weight functions [Dou Fangfang]. Similarly, spectral methods based on eigenfunction expansions \cite{Wang2023,Liu2013,Chen2020,Le2013,Wen2025} are extendable to CGLE systems. Following the theoretical approach of heat equations, we address the following inverse problems for CGLE.

In this paper, setting $a,b\in \mathbb{R}$ with $a> 0,b\neq0$, we consider a typical inverse problem which is to identify the source $f(x)$ in the following initial-boundary value problem for complex Ginzburg-Landau equation
\begin{equation}\label{eq-gov}
\begin{cases}
u_{t}-(a+ib)\Delta u = g(t) f(x), & (x,t) \in \Omega_T:=\Omega\times(0,T], \\
u(x, t)=0, & (x,t) \in \partial \Omega\times(0,T], \\
u(x, 0)=u_0(x), & x \in \Omega
\end{cases}
\end{equation}
from the measurement
\begin{equation}\label{eq-ob}
u(x, T_j) = h_j(x), \quad x \in \Omega_0 \subset \Omega\subset \mathbb R^d,
\end{equation}
where $u$ is a complex-valued function of $(x,t)\in\Omega\times(0,T]$, $\Omega_0$ be an arbitrary subdomain of $\Omega$, $0<T_1\le T_2 \le T$, $g(t):=e^{-\gamma t}$ with $\gamma\in \mathbb C$ and $h_j(x)$, $j=1,2$ are given. 
Here $\Omega\subset\mathbb R^d$, $d\ge1$, is a bounded domain with $\partial\Omega$ being an analytic hyper curve. 

This inverse problem holds significant value in both theoretical research and practical applications. For example, at the final time, the effects of external disturbances (e.g., thermal noise, magnetic field fluctuations) have been dissipated by the system \cite{alberto2013super}. Consequently, observational data at this terminal stage better reflect the intrinsic role of the source term, rather than transient disturbances. Under suitable assumptions on $f$ and $u$, we can obtain conditional stability in the case when $\Omega_0 = \Omega$ and $T_1=T_2=T$.

Before giving the main results, we recall that the operator $-\Delta$ is symmetric uniformly elliptic with domain $\mathcal D(-\Delta)=H^2(\Omega) \cap H_0^1(\Omega)$. Let $\{\lambda_n,\varphi_n\}_{n=1}^\infty$ denote the Dirichlet eigensystem of $-\Delta$, where $0<\lambda_1 \le \lambda_2\le \cdots $ and the corresponding eigenfunctions $\{\varphi_n\}_{n=1}^\infty$ form an orthonormal basis of $L^2(\Omega)$. We are now ready to state the first main theorem.
\begin{theorem}\label{thm-stabi-global}
Suppose $u_0=0$, $g(t)=e^{-\gamma t}$ with $\gamma\not\in  \cup_{{k\in\mathbb Z},{n \in \mathbb N\setminus\{0\}}} \left\{a\lambda_n + i \left(b\lambda_n - \frac{2k\pi}T \right) \right\}$, and $(u,f)\in \left(H^1\left(0,T;L^2(\Omega) \right) \cap L^2 \left(0,T; H^2(\Omega) \right)
\right) \times  \left(H^2(\Omega)\cap H_0^1(\Omega) \right)$ is a pair of solutions of \eqref{eq-gov} that corresponds to the measurement data $u(\cdot,T)$ in $\Omega \subset \mathbb R^d$. If $\|f\|_{H^2(\Omega) \cap H_0^1(\Omega)} \le M$ for some positive constant $M$, then there exists a constant $C=C(a,b,d,\gamma,T,M,\Omega)$ such that
$$
\|f\|_{L^2(\Omega)} \le CM^{\frac{1}{2}}\|u(\cdot,T)\|_{L^2(\Omega)}^{\frac{1}{2}}.
$$
\end{theorem}
The above result resolves the issue that the source profile cannot be observed directly. However, as is known, it costs less to observe the considered system in a subdomain than to make observations in the whole domain $\Omega$. On the other hand, global observation is sometimes difficult or impossible to implement, see e.g., \cite{CL}. 
For this, we will consider the inverse problem with local regional data. In this case we need two fixed  instances $0<T_1<T_2\le T$ and assume that $\partial\Omega$ is an analytic hyper curve.

We are now ready to state the second main theorem.
\begin{theorem}\label{thm-unique-local}
For any given $u_0\in H_0^1(\Omega) \cap H^2(\Omega)$ and $g(t)=e^{-\gamma t}$ with $\gamma$ being not in the set $  \cup_{n \in \mathbb N\setminus\{0\}}
\left\{a\lambda_n + i b\lambda_n \right\}$, let $u_j \in H^1\left(0,T;L^2(\Omega) \right) \cap L^2 \left(0,T; H_0^1(\Omega) \cap H^2(\Omega) \right)$ be the solution to \eqref{eq-gov} with respect to functions $f_j\in L^2(\Omega)$, $j=1,2$. Then $u_j(\cdot,T_1) = u_j(\cdot,T_2)$, $j=1,2$ in $\Omega_0$ implies $f_1 = f_2$ in $\Omega$.
\end{theorem}
By the above uniqueness result, we can further see that when the data have some perturbations, the source term must have corresponding changes.
\begin{theorem}\label{thm-stabi-local}

Under the same assumptions as in Theorem \ref{thm-unique-local}, suppose that $(u_j,f_j)$ is a pair of solutions of our inverse source problem \eqref{eq-gov} and \eqref{eq-ob} corresponding to data $u_j(x,T_1)$ and $u_j(x,T_2)$, $x\in\Omega_0$, $j=1,2$. Then
$$
\|f_1-f_2\|_{\mathcal B} \le \sum_{j=1}^2 \|u_1(\cdot,T_j) - u_2(\cdot,T_j)\|_{L^2(\Omega_0)},
$$
where $\Omega_0$ is a subdomain of $\Omega$ and the norm $\|\cdot\|_{\mathcal B}$ will be given in \eqref{eq10v6}. 
\end{theorem}

As is known, inverse problems exhibit inherent sensitivity to data noise due to their ill-posedness. Deep learning embeds physical equations as constraints into the loss function, constructing an implicit regularization framework that explicitly mitigates ill-posedness of the inverse problems, which has become a key tool for studying inverse problems in partial differential equations. Zhang, Li, and Liu  \cite{Zhang2023} proposed a deep neural network-based method to solve the inverse source problem. Their approach introduced a loss function combining residual derivatives and regularized observation data to enhance the regularity of solutions and address the ill-posedness. However, the application of deep learning to inverse problems of CGLE remains in its early stages. Theoretical understanding and algorithmic development are still limited, and further systematic studies are needed in aspects such as model construction, convergence analysis, and numerical validation.

This is a first work focusing on solving inverse source problems for 
\eqref{eq-gov} and \eqref{eq-ob}  from regional observation at fixed times, particularly when the observation domain are subset of the whole domain $\Omega$. 
Here we list the key challenges and the main innovations. 
\begin{enumerate}
    \item For the inverse source problem with highly unstable local observations, we overcome ill-posedness caused by the subdomain $\Omega_0 \subsetneq \Omega$ by proposing two instants ($T_1 < T_2$) and a variational norm $\|\cdot\|_{\mathcal{B}}$, which significantly reduces the spatial observation requirement compared to classical full-domain methods.
    \item  To handle the complex-valued solution of the CGLE, we develop Complex-PINNs (C-PINNs), which maintains numerical stability, overcoming limitations of real-valued PINNs for dissipative systems.
    \item We design a multi-residual weighted loss function with adaptive weights, and we establish the first quantitative characterization of the joint error for the source term 
    $f(x)$ and solution $u(x,t)$ in the complex-valued Ginzburg-Landau equation, demonstrating its square-root dependence on the PINNs residuals and data noise level (Theorems \ref{thm-pinns-glocal} and \ref{thm-pinns-local} in Section \ref{sec-num}).  
\end{enumerate}

To prove the main theorems mentioned above, several technical lemmas and settings are needed, so we collect them in Section \ref{sec-pre}. Preparing all necessities, by the eigenfunction expansion argument, we will finish the proof of Theorem \ref{thm-stabi-global} in Section \ref{sec-global}. In the following Section \ref{sec-local}, with the $x-$analyticity of the solution to the complex Ginzburg-Landau equation in Lemma \ref{lem-forward}, we show the uniqueness result of the inverse problem from local data at two instants. As an application, an integral identity is presented with which the stability in Theorem \ref{thm-stabi-local} can be constructed using a suitable topology.
In Section \ref{sec-num}, an algorithm based on PINNs is designed to obtain a numerical solution to our inverse source problem, and some typical numerical experiments are tested to show the validity of the PINNs scheme in Section \ref{se6v25}. Finally, concluding remarks are given in Section \ref{sec-rem}.

\section{Preliminaries}\label{sec-pre}

In this section, we establish the well-posedness theory of CGLE on domain $\Omega$.

On the basis of the Dirichlet eigensystem of the operator $-\Delta$, we define the fractional power $(-\Delta )^\alpha$ for $\alpha >0$ as
$$
(-\Delta )^\alpha \varphi := \sum_{n=1}^\infty \lambda_n^\alpha \langle \varphi,\varphi_n \rangle \varphi_n,\quad \varphi\in \mathcal{D}((-\Delta)^\alpha),
$$
where $\langle \cdot, \cdot \rangle$ denotes the inner product in $L^2(\Omega)$ and $\mathcal{D}((-\Delta)^\alpha)$ is a Hilbert space defined by
$$
\mathcal{D}((-\Delta)^\alpha) = \left\{ \varphi\in L^2(\Omega); \sum_{n=1}^\infty \lambda_n^{2  \alpha } |\langle \varphi,\varphi_n \rangle|^2 < \infty\right\}
$$
with the norm
$$
\|\varphi\|_{\mathcal{D}((-\Delta)^\alpha )} := \|(-\Delta)^\alpha  \varphi\|_{L^2(\Omega)} = \left( \sum_{n=1}^\infty \lambda_n^{2 \alpha } |\langle \varphi,\varphi_n \rangle|^2 \right)^{\frac12}.
$$
On the basis of the above settings and employing basic techniques such as eigenfunction expansion, and referring to the methods in \cite{evans2010partial}, we can obtain the expression of the solution to \eqref{eq-gov} as
\begin{equation}\label{sol-u}
\begin{aligned}
u(x, t) =
& 
\sum_{n=1}^\infty \left\langle u_0, \varphi_{n}\right\rangle  \exp\left\{-(a+ib)\lambda_{n}t\right \}  \varphi_{n}(x)
\\
&
+\sum_{n=1}^\infty \left\langle f, \varphi_{n}\right\rangle \int_{0}^{t} \exp\left\{-(a+ib)\lambda_{n}(t-\tau)\right \} g(\tau) \, \mathrm{d}  \tau \varphi_{n}(x),
\end{aligned}
\end{equation}
and get the regularity estimate. Although such methods and conclusions are quite common, in order to ensure the completeness of this paper, we still provide a detailed proof for the well-posedness of the following general  initial-boundary value problem
\begin{equation}\label{eq-gov'}
\begin{cases}
u_{t}-(a+ib)\Delta u = F(x,t), & (x,t) \in \Omega_T:=\Omega\times(0,T], \\
u(x, t)=0, & (x,t) \in \partial \Omega\times(0,T], \\
u(x, 0)=u_0(x), & x \in \Omega . 
\end{cases}
\end{equation}
\begin{lemma}\label{lem-forward}
Letting $u_0 \in H_0^1(\Omega) \cap H^2(\Omega)$, $F\in L^2(\Omega_T)$, then the initial-boundary value problem \eqref{eq-gov'} admits a unique solution $u\in H^1 \left(0,T;L^2(\Omega) \right) \cap L^2 \left(0,T; H_0^1(\Omega) \cap H^2(\Omega) \right)$ and can be represented by 
\begin{equation}\label{sol-uvnew}
\begin{aligned}
u(x, t) =
& 
\sum_{n=1}^\infty \left\langle u_0, \varphi_{n}\right\rangle  \exp\left\{-(a+ib)\lambda_{n}t\right \}  \varphi_{n}(x)
\\
&
+\sum_{n=1}^\infty \left\langle f, \varphi_{n}\right\rangle \int_{0}^{t} \exp\left\{-(a+ib)\lambda_{n}(t-\tau)\right \} g(\tau) \, \mathrm{d}  \tau \varphi_{n}(x),
\end{aligned}
\end{equation}
 Moreover, the solution $u$ can be estimated as follows
\begin{equation*}
    \|u\|_{H^1 \left(0,T;L^2(\Omega) \right)} + \|u\|_{L^2 \left(0,T; H^2(\Omega) \right)} \le C\|u_0\|_{H^2(\Omega)} + C\|F\|_{L^2(\Omega_T)},
\end{equation*}
where the constant $C>0$ is independent of $T$, $u$, $u_0$ and $F$. If further assuming $F=0$ and the boundary $\partial\Omega$ is an analytic hyper curve, then  $x\mapsto v(x,\cdot)\in L^2(0,T)$ is analytic with respect to the variable $x\in\Omega$.
\end{lemma}
\begin{proof}
We will show that the function $u$ defined in \eqref{sol-uvnew} is the required solution to the initial-boundary value problem \eqref{eq-gov}. We denote the two series on the right hand side in \eqref{sol-uvnew} as $v,w$ respectively. Firstly, we consider the spatial regularity of the function $w$. A direct calculation yields that
\begin{align*}
\|-\Delta w\|_{L^2(\Omega)}^2
&= \left\|\sum_{n=1}^\infty  \lambda_n  \int_{0}^{t}  \exp\left\{-(a+ib)\lambda_{n}(t-\tau)\right \} F_n(\tau)  \, \mathrm{d} \tau \varphi_n\right\|_{L^2(\Omega)} ^2 \\
&=  \sum_{n=1}^\infty \lambda_n^2 \left(\int_{0}^{t}  \exp\left\{-(a+ib)\lambda_{n}(t-\tau)\right \}  F_n(\tau) \, \mathrm{d} \tau  \right)^2 .
\end{align*}
Here $F_n:=\left\langle F, \varphi_{n}\right\rangle$. From the Young inequality for the convolution, it follows that
\begin{align*}
\int_0^T\left\| -\Delta w(t)\right\|_{L^2(\Omega)}^2 \,\mathrm{d}t 
\le \sum_{n=1}^\infty\lambda_n^2 \left[\int_0^T \exp\left\{-\lambda_{n} t a\right \} \, \mathrm{d} t \right]^2 \int_0^T |F_n(t) |^2 \, \mathrm{d} t,
\end{align*}
which together with the fact that
$$
\int_0^T \exp\left\{-\lambda_{n}t a\right \} \,\mathrm{d} t = \frac{1-e^{-\lambda_n T a}}{\lambda_na},
$$
implies
\begin{align*}
\int_0^T\left\| -\Delta w(t)\right\|_{L^2(\Omega)}^2 
\, \mathrm{d} t 
\le  \sum_{n=1}^\infty \int_0^T F_n^2(\tau) \, \mathrm{d} \tau \frac{ \left(1-e^{-\lambda_n Ta} \right)^2}{a^2}  \le C\|F\|_{L^2(\Omega_T)}^2,
\end{align*}
that is, $\|w\|_{L^2 \left(0,T;H^2(\Omega) \right)} \le C\|f\|_{L^2(\Omega)}$. Next, we consider the temporal-regularity estimate. By a direct calculation, we take $t-$derivative on both sides of \eqref{sol-uvnew} and we find that
\begin{align*}
\partial_t w = & \sum_{n=1}^\infty  \partial_t \int_{0}^{t} \exp\left\{-(a+ib)\lambda_{n}(t-\tau)\right \} F_n(\tau) \, \mathrm{d}  \tau \varphi_{n}(x) \\
=&  \sum_{n=1}^\infty F_n(t) \varphi_n(x) 
+ (a+ib) \sum_{n=1}^\infty \lambda_n  \varphi_n(x) \int_0^t \exp\left\{-(a+ib)\lambda_{n}(t-\tau)\right \} F_n(\tau) \, \mathrm{d} \tau \\
= &  F + (a+ib)\Delta w.
\end{align*}
By taking the $L^2 \left(0,T;L^2(\Omega) \right)$-norm on both sides of the above equation and in view of the spatial-regularity estimate for $u$, we can obtain the following temporal regularity estimate:
\begin{align*}
\|\partial_t w\|_{L^2 \left(0,T;L^2(\Omega) \right)}
&\le \|F\|_{L^2(\Omega_T)} + \|-\Delta w\|_{L^2(\Omega_T)}
\le  C\|F\|_{L^2 (\Omega_T)}.
\end{align*}
Collecting all the above results, we see that
\begin{align*}
\|w\|_{H^1 \left(0,T;L^2(\Omega) \right)} + \|w\|_{L^2 \left(0,T;H^2(\Omega) \right)}
\le C\|F\|_{L^2(\Omega_T)}.
\end{align*}
It remains to estimate $\Delta v$ and $\partial_t v$. For this, similar to the above argument, it follows that
\begin{align*}
\|-\Delta v\|_{L^2(\Omega)}^2
=&   \sum_{n=1}^\infty \lambda_n^2 \langle u_0,\varphi_n\rangle^2 \left| \exp\left\{-(a+ib)\lambda_{n}t\right \} \right|^2 \\
\le& C\sum_{n=1}^\infty \langle -\Delta u_0,\varphi_n\rangle^2 \left( \exp\left\{-\lambda_{n}ta\right \} \right)^2
\le C\|u_0\|_{H^2(\Omega)}^2,
\end{align*}
and
\begin{align*}
\|\partial_t v\|_{L^2(\Omega)}^2
= \|(a+ib)\Delta v\|_{L^2(\Omega)} ^2 
\le C\|u_0\|_{H^2(\Omega)}^2.
\end{align*}
Collecting all the above estimates, we finish the proof of the well-posedness of the solution to the problem. The proof of the $x-$analyticity of the solution is extremely similar to the case of $b=0$, so we omit the proof. One can refer to the argument in Ch.2 of \cite{evans2010partial}.
\end{proof}
As an application of the above result, we consider the regularity estimate for the initial-boundary value problem
\begin{equation}\label{eq-gov''}
\begin{cases}
u_{t}-(a+ib)\Delta u = 0, & (x,t) \in \Omega_T, \\
u(x, t)=\phi(x,t), & (x,t) \in \partial \Omega\times(0,T], \\
u(x, 0)=u_0(x), & x \in \Omega.
\end{cases}
\end{equation}
Here $\phi\in H^1 \left(0,T;H^{\frac32}(\partial\Omega) \right)$. We construct $\Phi$ satisfies the following elliptic problem
\begin{align*}
\begin{cases}
 - \Delta_x \Phi(x,t) = 0 &\text{ in } \Omega_T, \\
\Phi(x,t) = \phi(x,t) &\text{ on } \partial \Omega \times (0, T).
\end{cases}
\end{align*}
From theories of elliptic equation, it is not difficult to check that $\Phi\in H^1 \left(0,T;H^2(\Omega) \right)$ with the regularity estimate $\|\Phi(\cdot,t)\|_{H^2(\Omega)} \le C\|\phi(\cdot,t)\|_{H^{\frac32}(\partial\Omega)}$, and moreover, we can see that $W= u-\Phi$ satisfies
$$
\begin{cases}
W_{t}-(a+ib)\Delta W = \partial_t \Phi(x,t), & (x,t) \in \Omega_T, \\
W(x, t)=0, & (x,t) \in \partial \Omega\times(0,T], \\
W(x, 0)=u_0(x) - \Phi(x,0), & x \in \Omega.
\end{cases}
$$
Consequently, we can get the estimate for $W$ by using the regularity estimate in Lemma \ref{lem-forward}, hence the well-posedness of the problem \eqref{eq-gov''} can be established. Indeed, we have
\begin{corollary}\label{coro-forward}
Letting $u_0\in H_0^1(\Omega) \cap H^2(\Omega)$ and $\phi\in H^1 \left(0,T;H^{\frac32}(\partial\Omega) \right)$, then the initial-boundary value problem \eqref{eq-gov''} admits a unique solution $u\in H^1 \left(0,T;L^2(\Omega) \right) \cap L^2 \left(0,T; H_0^1(\Omega) \cap H^2(\Omega) \right)$ and can be estimated as follows
\begin{equation*}
    \|u\|_{H^1 \left(0,T;L^2(\Omega) \right)} + \|u\|_{L^2 \left(0,T; H^2(\Omega) \right)} \le C\|\phi\|_{H^1 \left(0,T;H^{\frac32}(\partial\Omega) \right)} +C\|u_0\|_{H^2(\Omega)},
\end{equation*}
where the constant $C>0$ is independent of $T$, $u$ and $F$. If further assuming $F=0$ and the boundary $\partial\Omega$ is an analytic hyper curve, then  $x\mapsto v(x,\cdot)\in L^2(0,T)$ is analytic with respect to the variable $x\in\Omega$.

\end{corollary}

\section{Global data measurement}\label{sec-global}

In this section, letting  $g(t):=e^{-\gamma t}$,  $\Omega_0=\Omega$ and $T_1=T_2=T$, we consider the uniqueness and stability of the inverse problem from the measurement at a fixed time $T$. For this, we first give the following useful lemma.

\begin{lemma}\label{lem-exp-nonzero}
Letting $\lambda_n>0$ and $a,b\in \mathbb{R}$ be fixed constants and  $g=e^{-\gamma t}$ with $\gamma \neq a\lambda_n + i\left(b\lambda_n - \frac{2k\pi}T \right) $ for any $k\in\mathbb Z,n \in \mathbb N\setminus\{0\}$, then
$$
\int_{0}^{T} \exp\left\{-(a+ib)\lambda_n (T-\tau)\right \} e^{-\gamma \tau}  \, \mathrm{d}  \tau \neq 0.
$$
\end{lemma}
\begin{proof}
In the case of $\gamma = (a+ib)\lambda_n$, it is trivial that
$$
\left|\int_{0}^{T} \exp\left\{-(a+ib)\lambda_{n}(T-\tau)\right \} e^{-\gamma \tau}  \, \mathrm{d}  \tau \right| = T \left|e^{-(a+ib)\lambda_{n}T} \right| = T\left|e^{-a\lambda_{n}T} \right|\neq 0.
$$
We only need to show the above assertion in the case of  $\gamma \neq (a+ib) \lambda_{n}$. In this case, we see that
\begin{align*}
&\int_{0}^{T} \exp\left\{-(a+ib)\lambda_{n}(T-\tau)\right \} e^{-\gamma \tau} \, \mathrm{d} \tau  \\
=& \exp\left\{-(a+ib)\lambda_{n} T \right \} \frac{\exp\{ ((a+ib)\lambda_n -\gamma)T\} - 1 }{(a+ib)\lambda_n - \gamma}. 
\end{align*}
If the above equation vanishes, then
$$
\exp\{ ((a+ib)\lambda_n -\gamma)T\} = 1,
$$
which implies that
$$
((a+ib)\lambda_n -\gamma) T = 2k\pi i, \quad k\in\mathbb Z,n \in \mathbb N\setminus\{0\}.
$$
We find that
$$
\gamma= a\lambda_n + i \left(b\lambda_n - \frac{2k\pi}T  \right),\quad k\in\mathbb Z,n \in \mathbb N\setminus\{0\},
$$
which contradicts to the assumption for $\gamma$. This completes the proof of the lemma.
\end{proof}
\begin{corollary}
Under the same assumptions in Lemma \ref{lem-exp-nonzero}, we suppose the function $u$ satisfying $f\in H^1\left(0,T;L^2(\Omega) \right) \cap L^2 \left(0,T; H_0^1(\Omega) \cap H^2(\Omega) \right)$ is a solution to \eqref{eq-gov} with $u_0=0$. Then $u(\cdot,T) =0$ in $\Omega$ implies $f = 0$ in $\Omega$.
\end{corollary}
\begin{proof}
From the representation formula \eqref{sol-u}, and noting  $u(x,T)=0$, $x\in\Omega$, it follows that
\begin{equation*}
\sum_{n=1}^\infty f_{n} \int_{0}^{T} \exp\left\{-(a+ib)\lambda_{n}(T-\tau)\right \} e^{-\gamma \tau}  \, \mathrm{d}  \tau \varphi_{n}(x)=0,\quad x\in\Omega,
\end{equation*}
which combined with the results in Lemma \ref{lem-exp-nonzero} implies that $f_n=0$ for $n \in \mathbb{N}^{+}$, i.e., $f(x)=0$ in $L^{2}(\Omega)$. Thus, we get the uniqueness of the inverse source problem from the terminal observation data.
\end{proof}
By slightly modifying of Lemma \ref{lem-exp-nonzero}, we can obtain the stability of the inverse source problem. The detailed steps are as follows. We first conclude from \eqref{sol-u} that
\begin{equation*}
u(x, T) = \sum_{n=1}^\infty f_{n} \int_{0}^{T} \exp\left\{-(a+ib)\lambda_{n}(T-\tau)\right \} g(\tau)  \, \mathrm{d}  \tau \varphi_{n}(x).
\end{equation*}
Now in view of Lemma \ref{lem-exp-nonzero}, by taking scalar product with $\varphi_n$ on both sides of the above equation, we can solve $f_n$ by
\begin{equation}\label{eq-fn}
f_{n} = \frac{\left\langle u(\cdot, T),\varphi_n \right\rangle }{ \int_{0}^{T} \exp\left\{-(a+ib)\lambda_{n}(T-\tau)\right \} g(\tau)  \, \mathrm{d} \tau}.
\end{equation}
Therefore, we get the representation of the function $f$ as follows:
\begin{equation*}
f=\sum_{n=1}^\infty f_{n} \varphi_n
= \sum_{n=1}^\infty \frac{\left\langle u(\cdot, T),\varphi_n \right\rangle \varphi_n }{ \int_{0}^{T} \exp\left\{-(a+ib)\lambda_{n}(T-\tau)\right \} g(\tau)  \, \mathrm{d} \tau}.
\end{equation*}
Taking the $L^2(\Omega)$-norm on both sides of the above equation, we conclude that
$$
\|f\|_{L^{2}(\Omega)}^2 = \sum_{n=1}^\infty \frac{\left|\langle u(\cdot, T),\varphi_n \rangle \right|^2}{\left|\int_{0}^{T} \exp\left\{-(a+ib)\lambda_{n}(T-\tau)\right \} g(\tau)  \, \mathrm{d} \tau \right|^2}.
$$
To give the upper bound of $\|f\|_{L^2(\Omega)}^2$, we need the following result.
\begin{lemma}\label{lem-lower}
Letting $g(t)=e^{-\gamma t}$ with $\gamma \not\in  \left\{a\lambda_n + i \left(b\lambda_n - \frac{2k\pi}T \right) \right\}$, $k\in\mathbb Z,n \in \mathbb N\setminus\{0\}$, then there exists a positive constant $C=C(T,\gamma,\Omega)$ such that
$$
\left|\int_{0}^{T} \exp\left\{-(a+ib)\lambda_{n}(T-\tau)\right \} g(\tau) \, \mathrm{d} \tau \right| \ge C\lambda_n^{-1}.
$$
\end{lemma}
\begin{proof}
In view of the assumption that $\gamma \not\in \cup_{k\in\mathbb Z,n \in \mathbb N\setminus\{0\}}\left\{a\lambda_n + i \left(b\lambda_n - \frac{2k\pi}T \right) \right\}$, by a direct calculation, we find
$$
I_n:=\int_{0}^{T} \exp\left\{-(a+ib)\lambda_{n}(T-\tau)\right \} e^{-\gamma \tau} \,  \mathrm{d}  \tau =  \frac{\exp\{ -\gamma T\} - \exp\left\{-(a+ib)\lambda_{n} T \right \} }{(a+ib)\lambda_n - \gamma},
$$
from which we further conclude from the inequalities 
$$
|(a+ib)\lambda_n - \gamma| \le \sqrt{ a^2+b^2 }\lambda_n + |\gamma| 
\le \lambda_n  \left(\sqrt{ a^2+b^2 } + \frac{|\gamma|}{\lambda_1} \right)
$$ 
that
\begin{align*}
I_n\ge  C\lambda_n^{-1} \left|\exp\{ -\gamma T\} - \exp\left\{-(a+ib)\lambda_{n} T \right \}\right|.
\end{align*}
Now from the assumption $\gamma\not\in \cup_{k\in\mathbb Z,n \in \mathbb N\setminus\{0\}} \left\{a\lambda_n + i \left(b\lambda_n - \frac{2k\pi}T \right) \right\}$, we see that
$$
\exp\{ -\gamma T\} - \exp\left\{-(a+ib)\lambda_{n} T \right \} \neq0,\quad \forall n=1,2,\cdots
$$
in view of Lemma \ref{lem-exp-nonzero}. Moreover, by noting that $a>0$ implies the sequence $\exp\{ -\gamma T\} - \exp\left\{-(a+ib)\lambda_{n} T \right\}$ admits limit $|\exp\{ -\gamma T\}|$, there exists a positive constant $C>0$ which is independent of $n$ such that the following estimate is valid
$$
\left| \exp\{ -\gamma T \} - \exp\left\{-(a+ib)\lambda_{n} T \right \} \right| \ge C>0,\quad \forall n=1,2,\cdots.
$$
Consequently, we obtain $I_n \ge  C\lambda_n^{-1}$. This completes the proof of the lemma.
\end{proof}
Now, from the lower bound in the above lemma, we can derive the stability of our inverse problem.
\begin{proof}[\bf Proof of Theorem \ref{thm-stabi-global}]
Using the Cauchy-Schwarz inequality, we have
$$
\begin{aligned}
\|f\|_{L^2(\Omega)}^2 = \sum_{n=1}^\infty \langle f,\varphi_n \rangle^2 
\le \left(\sum_{n=1}^\infty \lambda_n^{-2} |\langle f,\varphi_n \rangle|^{2} \right)^{\frac{1}{2}} \left(\sum_{n=1}^\infty \lambda_n^{2} |\langle f,\varphi_n \rangle|^{2} \right)^{\frac{1}{2}} ,
\end{aligned}
$$
which together with the assumption $\| f\|_{H^2(\Omega) \cap H_0^1(\Omega)} \le M$ further implies that
$$
\begin{aligned}
\|f\|_{L^2(\Omega)}^2
\le M \left(\sum_{n=1}^\infty  \lambda_n^{-2} |\langle f,\varphi_n \rangle|^{2} \right)^{\frac{1}{2}}.
\end{aligned}
$$
Moreover, by noting the identity \eqref{eq-fn} and the lower bound in Lemma \ref{lem-lower}, we see that
$$
\begin{aligned}
\|f\|_{L^2(\Omega)}^2
\le& M \left(\sum_{n=1}^\infty \lambda_n^{-2} \left|\frac{\left\langle u(\cdot, T),\varphi_n \right\rangle }{ \int_{0}^{T} \exp\left\{-(a+ib)\lambda_{n}(T-\tau)\right \} e^{-\gamma\tau}  \, \mathrm{d} \tau} \right|^{2} \right)^{\frac{1}{2}}
\\
\le& C M \left(\sum_{n=1}^\infty   \left| \left\langle u(\cdot, T),\varphi_n \right\rangle  \right|^{2} \right)^{\frac{1}{2}}  
= C M \left\| u(\cdot, T)\right\|_{L^2(\Omega)}.
\end{aligned}
$$
We finally obtain $\|f\|_{L^2(\Omega)} \le CM^{\frac{1}{2}}\|u(\cdot,T)\|_{L^2(\Omega)}^{\frac{1}{2}}$, which completes the proof of the theorem.
\end{proof}

\section{Local data measurement}\label{sec-local}

In this section, we will consider the domain $\Omega$ with analytic boundary $\partial\Omega$. In the sequel, we always assume that $g(t)=e^{-\gamma t}$ with a constant $\gamma\in\mathbb C$ satisfying $\gamma \neq \lambda_n(a+ib),\ n \in \mathbb N\setminus\{0\}$.

\subsection{Uniqueness from local data at two instants}

We first show that the local measurement data at two fixed instants is sufficient to determine the source profile in \eqref{eq-gov} by noting the $x-$analyticity of the solution to the problem \eqref{eq-gov} with $gf=0$ in Lemma \ref{lem-forward}. 

\begin{proof}[\bf Proof of Theorem \ref{thm-unique-local}]
As is known, the problem under consideration is linear, so it remains to show that $u(\cdot,T_1)=u(\cdot,T_2)=0$ in $\Omega_0$ implies that $f=0$ under the assumption $u_0=0$. We assume $u$ is a solution to the problem \eqref{eq-gov} with $u_0=0$ and $f\in L^2(\Omega)$. By \eqref{sol-u}, we know that
\begin{equation}\label{eq-AB}
\begin{aligned}
u(x, t)=&e^{-\gamma t} \sum_{n=1}^\infty f_{n} \frac{\varphi_n(x)}{(a+ib)\lambda_n -\gamma} -\sum_{n=1}^\infty f_{n} \frac{ \varphi_n(x)}{(a+ib)\lambda_n -\gamma} \exp \left\{-(a+ib)\lambda_n  t \right\} \\
:=&e^{-\gamma t} B(x)-A(x, t).
\end{aligned}
\end{equation}
For any fixed $t>0, A(\cdot, t)$ is real analytic with respect to $x \in \Omega$. In fact, it is easy to verify that $A(x, t)$ solves
$$
\begin{cases}
A_{t}(x,t)- (a+ib)\Delta A(x,t) = 0, & (x,t) \in \Omega_T, \\
A(x, t)=0, & (x,t)\in\partial\Omega\times(0,T), \\
A(x, 0)=\sum_{n=1}^\infty f_{n} \frac{\varphi_n(x)}{(a+ib)\lambda_n -\gamma} , & x \in \Omega.
\end{cases}
$$
For any $f \in L^{2}(\Omega)$, the initial value $A(x, 0) \in H^{2}(\Omega)\cap H_0^1(\Omega)$, we see that the solution $A(\cdot, t)$ for any fixed $t>0$ is analytic in $\Omega$ by Lemma \ref{lem-forward}. 

For local measurement data given by \eqref{eq-ob}, it follows from \eqref{eq-AB} that
\begin{equation}\label{eq-Phi}
\begin{aligned}
&A\left(x, T_{2}\right) e^{-\gamma T_{1}}-A\left(x, T_{1}\right) e^{-\gamma T_{2}}\\
=&e^{-\gamma T_{2}} h_{1}(x)-e^{-\gamma T_{1}} h_{2}(x)
:=\Phi(x), \  x \in \Omega_0.
\end{aligned}
\end{equation}
Since $A\left(\cdot, T_{i}\right)$ for $i=1,2$ are analytic, $\Phi(x)$ is also analytic in $\Omega$, which enables us to extend it analytically to the whole domain $\Omega$. Consequently, \eqref{eq-Phi} becomes
$$
A\left(x, T_{2}\right) e^{-\gamma T_{1}}-A\left(x, T_{1}\right) e^{-\gamma T_{2}}=\Phi(x), \quad x \in\Omega.
$$
From \eqref{eq-Phi}, we know that $h_{1}(x)=h_{2}(x)=0$ in $L^{2}\left(\Omega_0\right)$ implies $\Phi(x)=0$ for $x\in\Omega_0$. 
Based on the series expression of $A(x, t)$, we finally get
$$
 \frac{f_{n}}{(a+ib)\lambda_n -\gamma}\left[e^{-\left(\lambda_n T_2 (a+ib) + \gamma T_{1}\right)}-e^{-\left(\lambda_n  T_{1} (a+ib) + \gamma T_{2}\right)}\right]=\Phi_{n}:=\langle\Phi, \varphi_n\rangle.
$$
Consequently $\Phi(x)=0$ in $L^{2}(\Omega)$, which means $\Phi_{n}=0$ for $n=1, 2, \ldots$. Then we have $f_{n}=0$ for $n \in \mathbb{N}^{+}$. That is, the solution of our inverse problem is unique in $L^{2}(\Omega)$. This completes the proof of the theorem.
\end{proof}

\subsection{Conditional stability} 
In this subsection, following the arguments used in Li and Yamamoto \cite{LiYa06}, we will rely on the uniqueness result in the last subsection to construct a special topological structure and prove that the inverse source problem is stable in this topology. For this, we will give an integral identity with the aid of an adjoint problem, which reflects a corresponding relation of variations of the unknown source functions with changes of the initial-boundary values and additional observations.
\begin{lemma}[Integral identity]\label{le4.2v5}
Let $u$ be the solution of \eqref{eq-gov} with $u_0=0$, and $v$ be the solution to the dual problem of \eqref{eq-gov}, i.e.
\begin{equation}\label{eq-dual}
\begin{cases}
 \partial_t v + (a+ib)\Delta v = 0 & \text{ in }\Omega_T,\\
v=0 & \text{ on }\partial \Omega\times(0,T), \\
v(x, T) = \varphi(x) & \text{ in }\Omega,
\end{cases}
\end{equation}
where $\varphi\in C_0^\infty (\Omega_0)$, then the following integral identity holds:
\begin{align*}
\int_{\Omega_T} g(t) v(x, t) f(x) \,\mathrm{d}t \mathrm{d}x = \int_{\Omega_0}  u(x, T) \varphi(x) \,\mathrm{d}x .
\end{align*}
\end{lemma}

\begin{proof}

By integral by parts, and noting $u(x, 0) = 0$ and $v(x,T)=\varphi(x)$, we have
\begin{align*}
&\int_{\Omega_T} g(t) f(x) v(x, t) \,\mathrm{d}t \mathrm{d}x 
= \int_{\Omega_T} (\partial_t u - (a+ib)\Delta u ) v(x, t) \,\mathrm{d}t \mathrm{d}x  \\
=& - \int_{\Omega_T} u ( \partial_t  v +(a+ib) \Delta v) \,\mathrm{d}t \mathrm{d}x + \int_\Omega  u(x, T) \varphi( x) \,\mathrm{d}x,
\end{align*}
which together with the equation \eqref{eq-dual} of $v$ implies 
\begin{align*}
\int_\Omega \int_0^T g(t)  f(x) v(x, t) \,\mathrm{d}t \mathrm{d}x = \int_{\Omega_0}  u(x, T) \varphi(x) \,\mathrm{d}x.
\end{align*}
This completes the proof of the lemma.
\end{proof}

\begin{lemma}\label{lem-dualnorm}
We define the functional in the following
\begin{align}\label{eq10v6}
\|f\|_{ \mathcal{B}} : = \sup\limits_{ \substack{ 
  \varphi_j \in C_0^\infty ( \Omega_0), \, j=1,2\\
  \|\varphi_j \|_{L^2( \Omega_0)}  = 1
  }  }
\sum_{j=1}^2 \bigg| \int_\Omega \int_0^{T_j} g(t) v_j (x, t) f (x) \,\mathrm{d}t  \mathrm{d}x  \bigg|.
\end{align}
Here $v_j$ is defined by Lemma \ref{le4.2v5} with $\varphi=\varphi_j$, $j=1.2$. Then the functional $\| \cdot \|_{\mathcal{B}}$ is a new norm of the space $L^2( \Omega)$.
\end{lemma}

\begin{proof}
It is not difficult to check the equality
\begin{align*}
\|c f \|_{\mathcal{B}} = |c | \|f \|_{\mathcal{B}},
\end{align*}
and the triangle inequality
\begin{align*}
\|f_1 + f_2 \|_{\mathcal{B}} \le \|f_1 \|_{\mathcal{B}} + \|f_2 \|_{\mathcal{B}}.
\end{align*}
It remains to prove $f=0$ if $\|f\|_{\mathcal B}=0$. For this,  it is easily seen that
\begin{align*}
 \int_\Omega \int_0^{T_1} g \varphi_1 f \,\mathrm{d}t \mathrm{d}x = \int_\Omega \int_0^{T_2} g \varphi_2 f \,\mathrm{d}t \mathrm{d}x = 0, 
\end{align*}
since $\|f\|_{\mathcal{B}} = 0 $. Now by Lemma \ref{le4.2v5}, we can obtain 
$$
\int_{\Omega_0}  u(x, T_j) \varphi(x) \,\mathrm{d}x  = 0,\quad j=1,2,
$$
which implies that $u(x,T_1)=u(x,T_2)=0$. Finally, from the uniqueness result in Theorem \ref{thm-unique-local}, we get $f \equiv 0$. This completes the proof of the lemma.
\end{proof}
Now we are ready to give a proof of Theorem \ref{thm-stabi-local}. 
\begin{proof}[Proof of Theorem \ref{thm-stabi-local}]
By Lemma \ref{le4.2v5} and Lemma \ref{lem-dualnorm}, we have
\begin{equation*}
\begin{aligned}
\|{f_1-f_2}\|_{\mathcal{B}} &= \sup\limits_{ \substack{  
  \varphi_j \in C_0^\infty ( \Omega_0), \, j=1,2\\
  \|\varphi_j \|_{L^2(\Omega_0)}  = 1
  }  }
\sum_{j=1}^2 \bigg| \int_\Omega \int_0^{T_j} g(t) v_j (x, t) (f_1(x)-f_2(x)) \,\mathrm{d}t  \mathrm{d}x  \bigg|\\
&=\sup\limits_{ \substack{  
  \varphi_j \in C_0^\infty ( \Omega_0),\, j=1,2\\
  \|\varphi_j \|_{L^2( \Omega_0)}  = 1}  }
\sum_{j=1}^2 \bigg| \int_{\Omega_0}  (u_1(x,T_j) - u_2(x,T_j))\varphi_j (x, t)  \, \mathrm{d}x  \bigg|.
\end{aligned}
\end{equation*}
Here $u_k(x,t)$ are the solutions to \eqref{eq-gov} with respect to $f_k$, $k=1,2$. On the basis of the above calculation, we finally see that
\begin{equation*}
\begin{aligned}
\|{f_1-f_2}\|_{\mathcal{B}} 
\le \|u_1(\cdot,T_1) - u_2(\cdot,T_1)\|_{L^2(\Omega_0)} + \|u_1(\cdot,T_2) - u_2(\cdot,T_2)\|_{L^2(\Omega_0)}.
\end{aligned}
\end{equation*}
The proof of the theorem is complete.
\end{proof}

\section{The framework of the PINNs}\label{sec-num}
In this section, based on the physics-informed neural networks (PINNs) framework, we study the final observation to recover the unknown source $f(x)$ and the global distribution $u(x,t)$ by two deep neural networks. Since the solution $u(x,t)$ to \eqref{eq-gov} is complex valued. Therefore, we design two neural networks: one is a neural network with complex-valued input and output layers, the other is a neural network with real-valued input and output layers, and get $u_{NN}$ and $f_\eta$ as the neural network approximation of $u$ and $f$. The details are in the following subsections.

\subsection{The C-PINNs method}

For $u_{NN}$, we need to deal with its real and imaginary parts at the same time when applying the deep neural network. But the classical PINNs are quite complex when dealing with this kind of problem with complex value. To overcome this difficulty, we propose physics-informed neural networks based on the complex-valued structure, which is called Complex-PINNs (C-PINNs). This method replaces the standard layers by the complex one, which greatly enhance the ability of the classical PINNs in dealing with the complex-valued data. In particular, when we put the input data $z  = z_1 + i z_2$ in the neural networks, the weight $W$ and the bias $d$ of the input data are generated in the complex-valued layers, where $W$ and $d$ are complex numbers. The $k-$th complex linear layer $s_k( \cdot )$ is defined by
\begin{equation}
  \begin{cases}
s_k(z)=s_{k_1}(z)+ i s_{k_2}(z), \\
s_{k_1}(z) := z_1 W_{k_1} - z_2 W_{k_2} + d_{k_1},  \\
s_{k_2}(z) := z_1 W_{k_2} + z_2 W_{k_1} + d_{k_2},
\end{cases}
\nonumber
\end{equation}
where \(a_{k_1}\) and \(a_{k_2}\) denote the real and imaginary parts of the weight of the $k-$th linear layer, \(b_{k_1}\) and \(b_{k_2}\) represent the of real and imaginary parts of the bias of the $k-$th linear layer. Before moving in the $(k+1)-$th linear layer, we need to apply an activation function \( \sigma(\cdot) \) on \( s_{k}(z) \), and the activation function is defined to be
\begin{equation}
l_{k}(z) =\sigma( s_{k}(z) ).
\nonumber
\end{equation}
Thus, for any positive integer \(K\in \mathbb{N}\), the neural network \( S_{NN}(z) \) with $K$ layers can be defined as follows
\begin{equation}
S_{NN}(z)=W_{K}l_{{K-1}}\circ\cdots\circ l_{1}(z)+d_{K},
\nonumber
\end{equation}
where $a_k$ represents the weight of the input data, $b_k$ represents the bias of the input data, $a_k$ and $b_k$ are complex numbers.

In this paper, we will provide the complexification of the input data $z$ before inputting it into the neural network because the input data $z$ is a real number.
The complexification of $z$ is defined by
\begin{equation}
\tilde{z}=z+ i {\varepsilon }, \nonumber
\end{equation}
where the value of $\varepsilon$ is taken to be $10^{-5}$. This kind of complexification can enhance the stability of data processing in the neural network. We will also use the real linear layer neural network to optimize $f_\eta$. 

\subsection{Loss function and error estimate}
In this subsection, we will complete the PINNs to establish the error estimates for the inverse problem for the complex Ginzburg-Landau equation by introducing a new loss function. The details are as follows:
\begin{itemize}
    \item Interior PDE residual
\begin{equation}
\mathcal{R}_{int,NN,\eta}(x,t) := \partial_tu_{NN}(x,t)-(a+ib)\Delta{u_{NN}(x,t)}-f_\eta(x)g(t), 
\quad (x,t)\in\Omega_T.
\nonumber
\end{equation}
\item Spatial boundary residual
\begin{equation}
\mathcal{R}_{sb,NN}(x,t):=u_{NN} (x,t),\quad(x,t)\in\partial\Omega\times(0,T).
\nonumber
\end{equation}
\item Temporal boundary (initial status) residual
\begin{equation}
\mathcal{R}_{tb,NN}(x):=u_{NN}(x,0),\quad x\in\Omega.
\nonumber
\end{equation}
\item Observational data residual
\begin{equation}
\mathcal{R}_{T_j,NN}(x):=u_{NN}(x,T_j)-h_j^{\delta}(x),\quad j=1,2,\ x\in\Omega_0.
\nonumber
\end{equation}
\end{itemize}
In the above expression, $\Omega_T := \Omega \times (0,T]$, $T_j$, $j=1,2$ represent the observation instants of the solution $u$, and
\begin{equation}
\|h_j^{\delta} - h_j\|_{L^2(\Omega)} \le \delta\nonumber
\end{equation}
represents the observation data with noise of level $\delta>0$. 

We will minimize these residuals comprehensively with some weights balancing different residuals to get better results for the complex Ginzburg-Landau equation. The following are the corresponding loss functions used in this paper:
\begin{equation}\label{eq-lose}
 \begin{aligned}
      \mathcal{L}_{\lambda,\beta}(NN,\eta)= 
      & \left\|\mathcal{R}_{T_1,NN} \right\|_{L^{2}(\Omega)}^{2} + \left\|\mathcal{R}_{T_2,NN} \right\|_{L^{2}(\Omega)}^{2}
      + \lambda \left\|\mathcal{R}_{int,NN,\eta} \right\|_{L^{2}(\Omega_{T})}^{2}
      \\
      &+\|\mathcal{R}_{tb,NN}\|_{H^{2}(\Omega)}^{2}+\beta \|\mathcal{R}_{sb,NN}\|_{H^1(0,T;H^\frac{3}{2}(\partial\Omega)) }^{2},
  \end{aligned}
\end{equation}
where $\lambda$ and $\beta $ are hyperparameters.

We next define the generalized error as follows:
\begin{equation*}
    \begin{cases}
    \mathcal{E}_{G,f} := 
    \left\|f_{\eta}^{*}-f \right\|_{L^2(\Omega)},\\
    \mathcal{E}_{G,u} := 
    \left\|u_{NN}^{*}-u \right\|_{C\left([0,T];L^2(\Omega)\right)},
    \end{cases}
\end{equation*}
where $u_{NN}^*$  and $f_{\eta}^*$ are the optimal solutions obtained by minimizing the loss function \eqref{eq-lose},  approximating the exact solution $(u, f )$ of \eqref{eq-gov}.

\begin{theorem}\label{thm-pinns-glocal}
In the case of $T_1=T_2=T$, for the approximate solution $ \left(u_{NN}^*, f_{\eta}^* \right)$ of the inverse problem \eqref{eq-gov}, we have the following error estimates:
\begin{equation*}
            \mathcal{E}_{G,f}\leq C\left(\mathcal{E}_{T}^*(NN,\eta)
            +\mathcal{E}_{int}^*(NN,\eta)
            +\mathcal{E}_{sb}^*(NN,\eta)
            +\mathcal{E}_{tb}^*(NN,\eta)
            +\delta\right)^{\frac{1}{2}},
    \end{equation*}
\begin{equation*}
            \mathcal{E}_{G,u}
            \leq C\left(\mathcal{E}_{T}^*(NN,\eta)
            +\mathcal{E}_{int}^*(NN,\eta)
            +\mathcal{E}_{sb}^*(NN,\eta)
            +\mathcal{E}_{tb}^*(NN,\eta)
            +\delta\right)^{\frac{1}{2}}.
    \end{equation*}
Here,
\begin{equation*}
\begin{aligned}
\mathcal{E}_{T}^*(NN,\eta):= \left\|\mathcal{R}^*_{{T,NN}} \right\|_{L^2(\Omega)}, 
\quad\mathcal{E}_{int}^*(NN,\eta):= \left\|\mathcal{R}^*_{int,NN,\eta} \right\|_{L^2(\Omega_T)},\\
\mathcal{E}_{sb}^*(NN,\eta):= \left\|\mathcal{R}^*_{sb,NN} \right\|_{H^1 \left(0,T;H^\frac{3}{2}(\partial\Omega) \right)} ,
\quad\mathcal{E}_{tb}^*(NN,\eta):= \left\|\mathcal{R}^*_{tb,NN} \right\|_{H^2(\Omega)},
\end{aligned}
\end{equation*}
the constant $C$ is positive and may depend on $a,b,d,\gamma,T,M,\Omega$.  
\end{theorem}

\begin{proof}
For the sake of convenience, we consider the case where the initial value $u_0$ is vanished in \eqref{eq-gov}, that is,
\begin{equation*}
\begin{cases}
u_t -(a+ib)\Delta u = g(t)f(x), & (x,t) \in \Omega_T, \\u(x, t) = 0, & (x,t) \in \partial\Omega\times(0,T], \, \\u(x, 0) = 0, & x \in \Omega.
\end{cases}
\end{equation*}
To indicate the dependence of the solution $u$ on $f$, we sometimes denote $u$ by $u[f]$. Letting $\hat{u} := u_{NN}^* -u[f]$, and noting the function $u_{NN}^*$ satisfies
\begin{equation*}
\mathcal{R}_{int,NN,\eta}^*(x,t) := \partial_tu_{NN}^*(x,t)-(a+ib)\Delta{u_{NN}^*(x,t)}-f_\eta^*(x)g(t),\quad (x,t)\in\Omega_T,
\end{equation*}
we can see that $\hat{u}$ satisfies the following initial-boundary value problem:
\begin{equation}\label{eq-govlose}
        \begin{cases}
        \partial_t\hat{u}-(a+ib)\Delta\hat{u} = \mathcal{R}_{int,NN,\eta}^*(x,t) + g(t) \left(f_\eta^*-f \right)(x), &(x,t)\in\Omega_T,\\
        \hat{u}(x,t)=\mathcal{R}_{sb,NN}^*(x,t),&(x,t) \in \partial\Omega\times(0,T],\\
        \hat{u}(x,0)=\mathcal{R}_{tb,NN}^*(x), &x\in\Omega,
        \end{cases}
    \end{equation}
with the following terminal conditions
\begin{equation*}
        \begin{aligned}
        \hat{u}(x, T_j) = u_{NN}^* (x, T_j) - u[ f ](x, T_j) =\mathcal{R}_{T_j,NN}^*(x) +h_j^\delta(x)- h_j(x),\quad j=1,2 .
        \end{aligned}
    \end{equation*}
Now by the superposition principle of 
\eqref{eq-govlose}, we make the decomposition $\hat{u} := \hat{u}_1 + \hat{u}_2$, where  $\hat{u}_1 $ is such that
\begin{equation}\label{eq-govlose-1}
        \begin{cases}
        \partial_t\hat{u}_1-(a+ib)\Delta\hat{u}_1= g(t) \left(f_\eta^*-f \right)(x),&(x,t)\in\Omega_T,\\
        \hat{u}_1(x,t)=0,&(x,t) \in \partial\Omega\times(0,T],\\\hat{u}_1(x,0)=0,&x\in\Omega,
        \end{cases}
    \end{equation}
with the conditions
\begin{equation}\label{eq-u_1}
\hat{u}_1(x,T_j)=\mathcal{R}_{T_j,NN}^*(x)-\left(h_j(x)-h_j^\delta(x) \right)-\hat{u}_2(x,T_j),
\end{equation}
and $\hat{u}_2$ satisfies the following equations
\begin{equation}\label{eq-govlose-2}
        \begin{cases}
        \partial_t\hat{u}_2-\Delta\hat{u}_2=\mathcal{R}_{int,NN,\eta}^*(x,t),&(x,t)\in\Omega_T,\\\hat{u}_2(x,t)=\mathcal{R}_{sb,NN}^*(x,t),&(x,t) \in \partial\Omega\times(0,T],\\\hat{u}_2(x,0)=\mathcal{R}_{tb,NN}^*(x),&x\in\Omega.
        \end{cases}
    \end{equation}
In the case of $T_1=T_2=T$, we apply Theorem \ref{thm-stabi-global} to \eqref{eq-govlose-1}-\eqref{eq-u_1} to yield 
\begin{equation*}
        \begin{aligned}
        &\left\|f_\eta^{*}-f \right\|_{L^2(\Omega)} 
        \leq CM^{\frac{1}{2}} \left\|\hat{u}_1(\cdot, T) \right\|_{L^2(\Omega)}^{\frac{1}{2}} \\
        =& CM^{\frac{1}{2}}\left( \left\|\mathcal{R}_{T,NN}^*- \left(h-h^\delta \right)-\hat{u}_2(\cdot,T) \right\|_{L^2(\Omega)} \right)^{\frac{1}{2}}\\
        \leq & CM^{\frac{1}{2}}\left(
        \left\|\mathcal{R}_{T,NN}^* \right\|_{L^2(\Omega)} +\delta+ \left\|\hat{u}_2(\cdot,T) \right\|_{L^2(\Omega)} \right)^{\frac{1}{2}},
        \end{aligned}
    \end{equation*}
this together with \eqref{eq-govlose-2}, Lemma \ref{lem-forward} and Corollary \ref{coro-forward} implies that
\begin{equation}\label{eq-f-ae-2}
\begin{aligned}
\left\|f_\eta^*-f \right\|_{L^2(\Omega)}
&\leq CM^{\frac{1}{2}} \Big( \left\|\mathcal{R}_{T,NN}^* \right\|_{L^2(\Omega)} +\delta+ \left\|\mathcal{R}_{int,NN,\eta}^* \right\|_{L^2(\Omega_T)}
\\
&\qquad+ \left\|\mathcal{R}_{sb,NN}^* \right\|_{H^1 \left(0,T;H^\frac{3}{2}(\partial\Omega) \right)}+ \left\|\mathcal{R}_{tb,NN}^* \right\|_{H^2(\Omega)} \Big)^{\frac{1}{2}}.
\end{aligned}
\end{equation}
Now we will give an error estimate for $\left\|\hat{u} \right\|_{H^1 \left(0,T;L^{2}(\Omega) \right)} = 
\left\|u_{NN}^{*}-u \right\|_{H^1\left([0,T];L^2(\Omega)
\right)}$.
According to the regularization theory of 
\eqref{eq-govlose}, see Lemma \ref{lem-forward} and Corollary \ref{coro-forward}, there holds
    \begin{equation*}
        \begin{aligned}
        & \left\|u_{NN}^{*}-u \right\|_{H^1\left(0,T;L^2(\Omega)\right)} 
       \\
        &
        \leq C \left( \left\|\mathcal{R}^{*}_{int,NN,\eta}+\left(f_\eta^*-f \right)g \right\|_{L^2(\Omega_T)}+
        \left\|\mathcal{R}^*_{sb,NN} \right\|_{H^1 \left(0,T;H^\frac{3}{2}(\partial\Omega) \right)}
        +  \left\|\mathcal{R}^*_{tb,NN} \right\|_{H^2(\Omega)} \right),
        \end{aligned}
    \end{equation*}
which combines with the estimate \eqref{eq-f-ae-2} implies that
\begin{equation*}
    \begin{aligned}
        &\left\|u_{NN}^{*}-u \right\|_{H^1\left(0,T;L^2(\Omega)\right)}\\ 
        \leq& C\left( \left\|\mathcal{R}^{*}_{int,NN,\eta} \right\|_{L^{2}(\Omega_{T})}
        + \left\|f_\eta^*-f \right\|_{L^{2}(\Omega)}
        + \left\|\mathcal{R}^*_{sb,NN} \right\|_{H^1 \left(0,T;H^\frac{3}{2}(\partial\Omega) \right)}
        + \left\|\mathcal{R}^*_{tb,NN} \right\|_{H^2(\Omega)}\right) \\
        \leq & C\Big( \left\|\mathcal{R}^{*}_{int,NN,\eta}
        \right\|_{L^{2}(\Omega_{T})}
        + \left\|\mathcal{R}_{sb,NN^*} \right\|_{H^1 \left(0,T;H^\frac{3}{2}(\partial\Omega) \right)}
        + \left\|\mathcal{R}^*_{tb,NN} \right\|_{H^2(\Omega)}\Big)
        \\
        & +\left( \left\|\mathcal{R}^*_{{T,NN}} \right\|_{L^2(\Omega)} +\delta
        + \left\|\mathcal{R}^*_{int,NN,\eta} \right\|_{L^2(\Omega_T)}
        + \left\|\mathcal{R}^*_{sb,NN} \right\|_{H^1 \left(0,T;H^\frac{3}{2}(\partial\Omega) \right)}
        + \left\|\mathcal{R}^*_{tb,NN} \right\|_{H^2(\Omega)} \right)^{\frac{1}{2}}.
    \end{aligned}
\end{equation*}
Noting that the inequality $\mathcal{R} \leq \mathcal{R}^\rho$  holds, where $ \rho\in(0,1)$ and the residual term satisfies $0\leq \mathcal{R}\leq 1$, we finally get the following inequality
    \begin{equation*}
        \begin{aligned}
        & \left\|u_{NN}^{*}-u \right\|_{H^1\left(0,T;L^2(\Omega)\right)}
        \\
        \leq  & 
        C \left( \left\|\mathcal{R}^*_{{T,NN}} \right\|_{L^2(\Omega)} +\delta
        + \left\|\mathcal{R}^*_{int,NN,\eta} \right\|_{L^2(\Omega_T)}+
        \left\|\mathcal{R}^*_{sb,NN} \right\|_{H^1 \left(0,T;H^\frac{3}{2}(\partial\Omega) \right)}
        + \left\|\mathcal{R}^*_{tb,NN} \right\|_{H^2(\Omega)} \right)^{\frac{1}{2}}.
        \end{aligned}
    \end{equation*}
We complete the proof of the theorem.
\end{proof}
In the case of $T_1<T_2$, by an argument similar to the above theorem with Theorem \ref{thm-stabi-global} being replaced by Theorem \ref{thm-stabi-local}, we have
\begin{theorem}\label{thm-pinns-local}
Assuming $T_1<T_2\le T$ and let $ \left(u_{NN}^*, f_{\eta}^* \right)$ be the approximate solution to \eqref{eq-gov} by the neural network, then the following error estimate 
\begin{equation*}
            \left\|f_\eta^* - f \right\|_{\mathcal B} \leq C\left(\mathcal{E}_{T}^*(NN,\eta)
            +\mathcal{E}_{int}^*(NN,\eta)
            +\mathcal{E}_{sb}^*(NN,\eta)
            +\mathcal{E}_{tb}^*(NN,\eta)
            +\delta\right),
    \end{equation*}
is valid, where 
\begin{equation*}
\begin{aligned}
\mathcal{E}_{T}^*(NN,\eta):= \sum_{j=1}^2 \left\|\mathcal{R}^*_{{T_j,NN}} \right\|_{L^2(\Omega_0)}, 
\quad\mathcal{E}_{int}^*(NN,\eta):=
\left\|\mathcal{R}^*_{int,NN,\eta} \right\|_{L^2(\Omega_T)},\\
\mathcal{E}_{sb}^*(NN,\eta):=
\left\|\mathcal{R}^*_{sb,NN} \right\|_{H^1 \left(0,T;H^\frac{3}{2}(\partial\Omega) \right)} ,
\quad\mathcal{E}_{tb}^*(NN,\eta):= 
\left\|\mathcal{R}^*_{tb,NN} \right\|_{H^2(\Omega)},
\end{aligned}
\end{equation*}
the constant $C$ is positive and may depend on $a,b,d,\gamma,T,\Omega$.  
\end{theorem}

\section{Numerical experiments}
\label{se6v25}
In this section, we will carry out numerical simulation of the neural network algorithm to verify its effectiveness for both one-dimensional and two-dimensional cases. To this end, we fix a neural network with $2$ hidden layers and $20$ neurons per layer to represent \(u_{NN}\) and \(f_\eta\). The activation function is \({\sigma}(x)=x\cdot{\tanh(x)}\) and the hyperparameters are \(\lambda=0.01\) and \(\beta=1\). The number of training epochs is \(3\times10^{5}\), the learning rate starts at $0.01$ and is halved every $3\times 10^4$ iterations. 
To evaluate the integrals in \eqref{eq-lose} numerically so that we can determine the parameters from the discrete training set, we introduce the following sets:
\begin{equation}
	\begin{aligned}
		&S_{d}:=\{(x_{n},T_j);x_{n}\in\Omega,\quad n=1,2,\cdots,N_{d}\}, \quad j=1,2,\\
		&\mathcal{S}_{int}:=\left\{ \left(\widetilde{x}_{n},\widetilde{t}_{n} \right); \left(\widetilde{x}_{n},\widetilde{t}_{n} \right)\in\Omega_{T},\quad n=1,2,\cdots,N_{int}\right\}, \\
		&\mathcal{S}_{tb}:=\left\{ \left(\overline{x}_{n},0 \right);\overline{x}_{n}\in\Omega,\quad n=1,2,\cdots,N_{tb}\right\}, \\
		&S_{sb}:=\left\{ \left(\widehat{x}_{n},\widehat{t}_{n} \right); \left(\widehat{x}_{n},\widehat{t}_{n} \right)\in\partial\Omega_{T},\quad n=1,2,\cdots,N_{sb}\right\}.
	\end{aligned}
	\nonumber
\end{equation}
Moreover, to evaluate the accuracy of our numerical results, we define the relative errors as follows:
\begin{equation*}
\begin{cases}
\text{Re}_u := \dfrac{\|u - u_{\text{exact}}\|_{L^2(\Omega_T)}}{\|u_{\text{exact}}\|_{L^2(\Omega_T)}}, \\[1em]
\text{Re}_f := \dfrac{\|f - f_{\text{exact}}\|_{L^2(\Omega)}}{\|f_{\text{exact}}\|_{L^2(\Omega)}}.
\end{cases}
\end{equation*}
Here, \(u\) and \(f\) represent the numerical solutions obtained from the algorithm, while \(u_{\text{exact}}\) and \(f_{\text{exact}}\) denote the corresponding exact solutions. The relative errors are computed using the \(L^2\) norms over the domains \(\Omega_T\) and \(\Omega\), respectively.

\subsection{One-dimensional example}
In this subsection, setting $\Omega=[0,\pi],a=1,b=1$ and $T=1$, we consider the one dimensional complex Ginzburg-Landau equation
\begin{equation*}
  \begin{cases}
u_{t}-(a+ib)\Delta u= e^{-t}f(x), & (x,t)\in\Omega_T,

\\ u(x,t)=0, & (x,t)\in\partial\Omega\times(0,T],

\\ u(x,0)=0, & x\in\Omega ,

\end{cases}
\end{equation*}
where $f(x)= \sin x$, and the exact solution $u$ of the above equation is
\[
u(x,t)=\frac{\sin x}{a+ib-\gamma }\left[e^{-t}-e^{-(a+ib)t} \right],\quad (x,t)\in\Omega_T.
\]
Let $h^{\delta}$ denote the observed data with noise $\delta>0$, defined as:
$$
h^{\delta}:=h(x)(1+\delta \cdot randn(0,1)),
$$
where $randn(0,1)$ represents a standard normal distribution.
\begin{example}{Global observation and experimental results.}
In the whole domain $\Omega=[0,\pi]$, taking $u(x,1),x\in\Omega$ to be the observation data for $u$, i.e.
$$
h^{\delta}(x)=u^{\delta}(x,1)=\frac{\sin(x)}{i} \left[e^{-1}-e^{-1-i} \right](1+\delta \cdot randn(0,1)).
$$
We consider four noise levels: $\delta_1 = 0$, $\delta_2 = 0.001$, $\delta_3 = 0.01$, and $\delta_4 = 0.1$, aiming to evaluate the adaptability and robustness of our proposed model as the observation accuracy changes.
\end{example}

As shown in Figure \ref{Fig.1D.Global.f} and Table \ref{Tab.1D.Global.f}, based on the data from Figure \ref{Fig.1D.Global.f} and Table \ref{Tab.1D.Global.f}, we observe that as the noise level $\delta$ increases, the relative error of the inverted source term $f$ also increases. However, this growth is still within an acceptable range. The experimental results indicate that at a certain fixed instant, the measurement data from the whole region are sufficient for accurately reconstructing the source function $f$ using the proposed method. 

\begin{figure}[!t]
    \centering
    \subfigure{
        \includegraphics[width=0.45\textwidth]{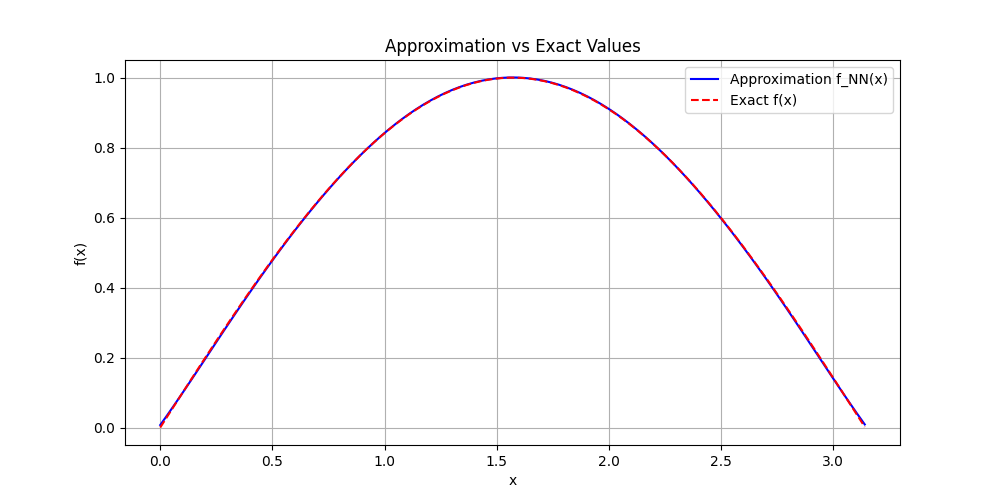}
    }
    \subfigure{
        \includegraphics[width=0.45\textwidth]{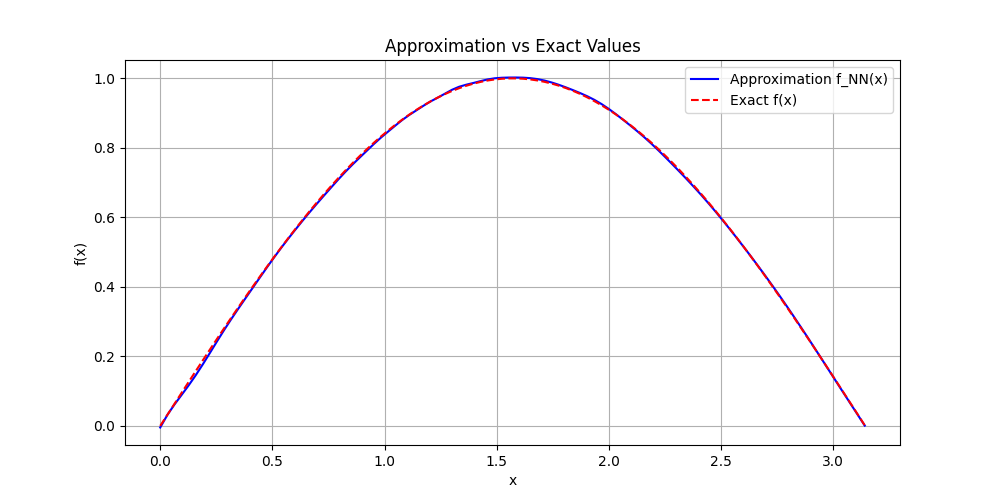}
    }
    \subfigure{
        \includegraphics[width=0.45\textwidth]{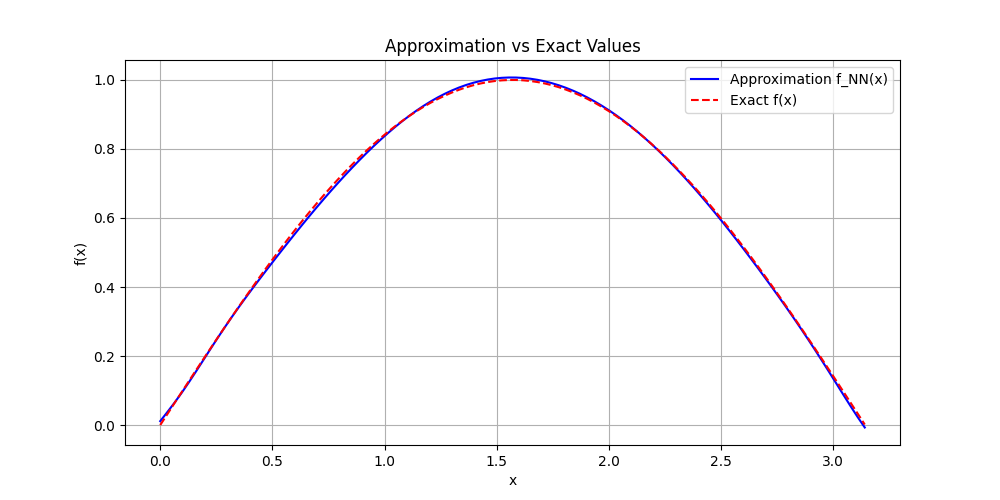}
    }
    \subfigure{
        \includegraphics[width=0.45\textwidth]{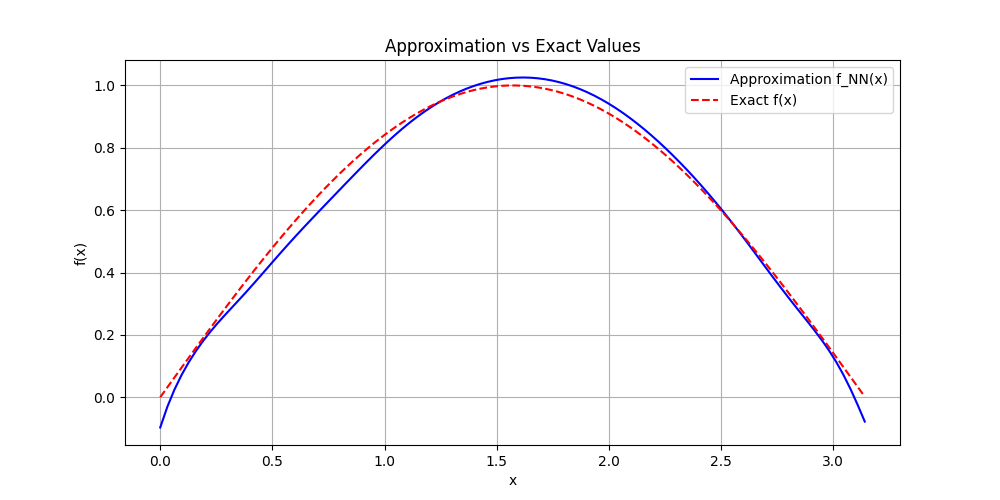}
    }
    \caption{The exact solution (red) and numerical approximation (blue) of the source term $f$ under different observation noise levels: top-left: $\delta_1=0$; top-right: $\delta_2=0.001$; bottom-left: $\delta_3=0.01$; bottom right: $\delta_4=0.1$.}
    \label{Fig.1D.Global.f}
\end{figure}

\begin{table}[!t]
    \centering
    \begin{tabular}{c|c}
        \hline
        \textbf{Noise Level} & $Re f$ \\
        \hline
        $\delta_1=0$     & 0.26916\% \\
        $\delta_2=0.001$ & 0.50045\% \\
        $\delta_3=0.01$  & 0.84765\% \\
        $\delta_4=0.1$   & 4.52661\% \\
        \hline
    \end{tabular}
    \caption{  The relative errors of the source term $f$ under different observation noise levels $\delta_1,\delta_2,\delta_3,\delta_4$.}
    \label{Tab.1D.Global.f}
\end{table}

\begin{example}{Local observation and experimental results.}

In the previous example, we conducted a numerical simulation for recovering the source term from global information; in this example, we investigate the numerical experiment of local observation data. We select the local observation region $\Omega_0\subset\Omega=[0,\pi]$ and choose $u=(x,0.3),u=(x,0.5),x\in\Omega_0$ as the observation regions for $u$. The noisy observations can be represented as:
$$
h^{\delta}_1(x)=u^{\delta}(x,0.3)=\frac{\sin(x)}{i}  \left[e^{-0.3}-e^{-0.3(1+i)} \right](1+\delta \cdot randn(0,1)),
$$
and 
$$
h^{\delta}_2(x)=u^{\delta}(x,0.5)=\frac{\sin(x)}{i} \left[e^{-0.5}-e^{-0.5(1+i)} \right](1+\delta \cdot randn(0,1)) ,
$$
 where the noise level $\delta=0.01$. We consider different observation regions $\Omega_0:$  $\Omega_{0}^{1}:=[0,0.5\pi],\Omega_{0}^{2}:=[0,0.1\pi],\Omega_{0}^{3}:=[0.3\pi,0.5\pi]$, and $\Omega_{0}^{4}:=[0,0.1\pi]\cup[0.9\pi,\pi]$.
    
\end{example}

The experimental details for local observations are consistent with those for global observations, except for the number of collocation points. In the local data measurement, we use $ N=N_{\mathrm{int}} + N_{\mathrm{sb}} + N_{\mathrm{tb}} + N_d = 256 + 256\times2 + 256 + 256\times2 = 1563$ as the number of collocation points, which are randomly sampled in four different domains, i.e., interior spatio-temporal domain, spatial and temporal boundary, additional measurement domain.

As shown in Figure \ref{Fig.1D.Local.f} and Table \ref{Tab.1D.Local.f}, analysis of the relative error data for the reconstructed source term $f$ in different local areas (Table \ref{Tab.1D.Local.f}) reveals that smaller observation domains correlate with larger relative errors in $f$. However, it is still within an acceptable range for us. We also found that when the size of the observation area is the same, and different specific observation ranges are selected, the relative errors of the inverted source term $f$ are similar, which shows that our inversion method has certain stability. Therefore, it can be known that in one-dimensional case, the local measurement data at two fixed instants is sufficient to accurately recover the source distribution $f$.

\begin{figure}[!t]
    \centering
    \subfigure{
        \includegraphics[width=0.45\textwidth]{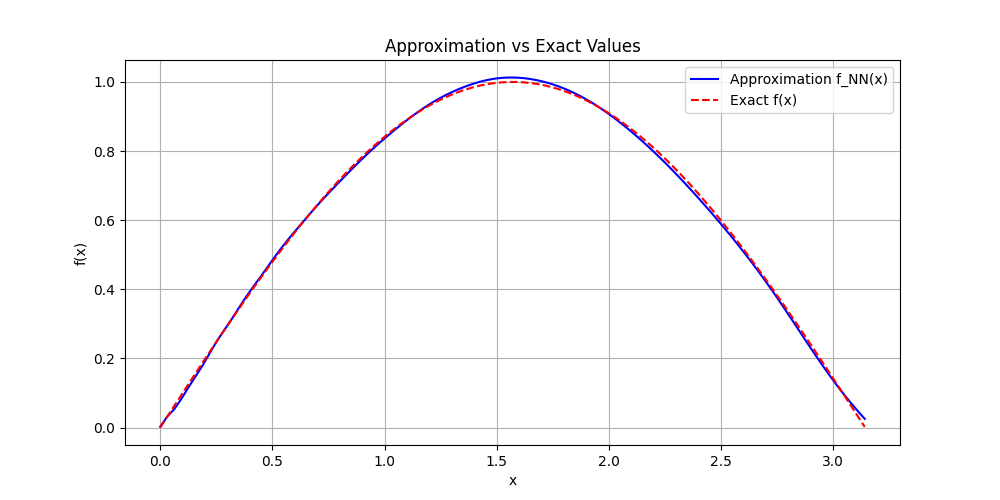}
    }
    \subfigure{
        \includegraphics[width=0.45\textwidth]{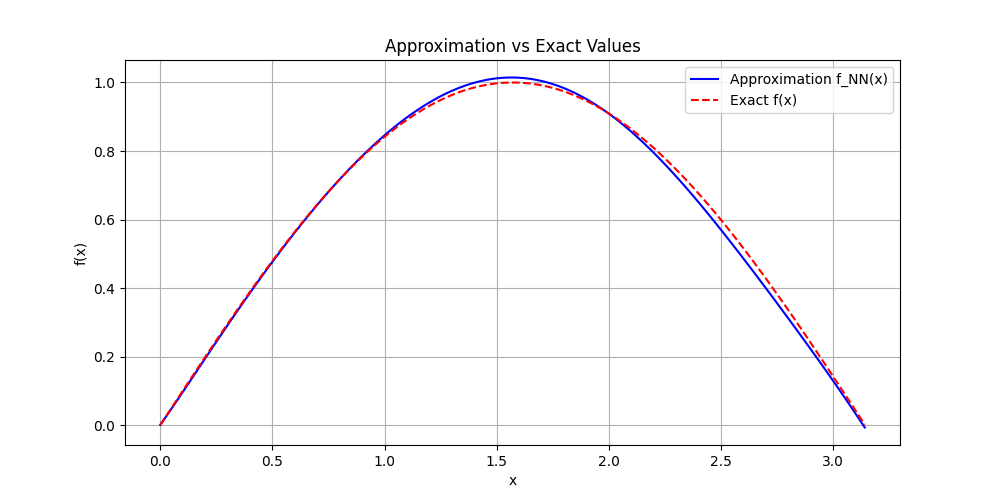}
    }
    \subfigure{
        \includegraphics[width=0.45\textwidth]{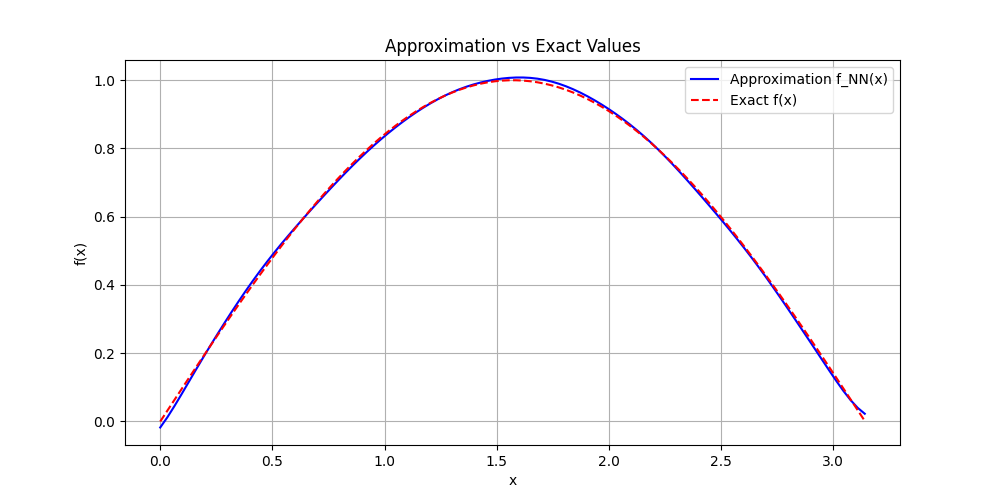}
    }
    \subfigure{
        \includegraphics[width=0.45\textwidth]{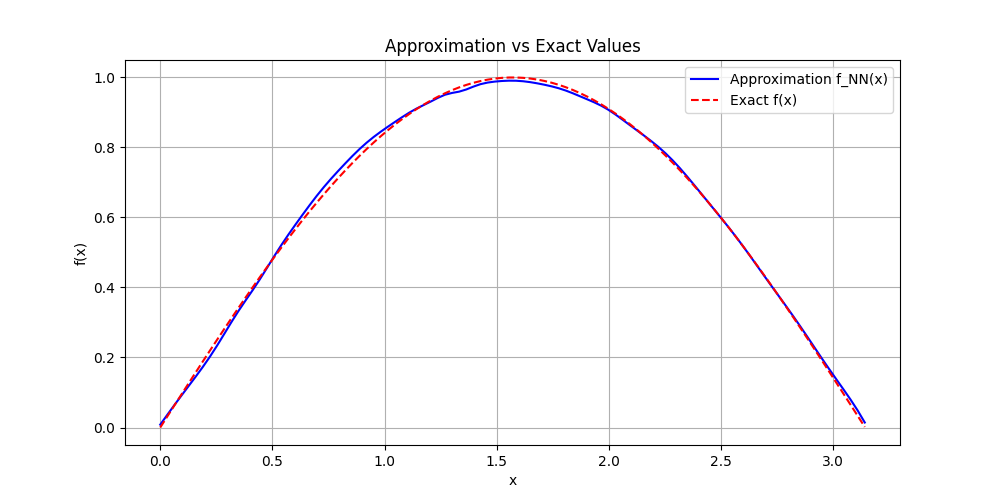}
    }
    \caption{The exact solution (red) and numerical approximation (blue) of the source term $f$ for different observation regions: top-left: $\Omega_{0}^{1}:=[0,0.5\pi]$; top-right: $\Omega_{0}^{2}:=[0,0.1\pi]$; bottom-left: $\Omega_{0}^{3}:=[0.3\pi,0.5\pi]$; bottom-right: $\Omega_{0}^{4}:=[0,0.1\pi]\cup[0.9\pi,\pi]$.}
    \label{Fig.1D.Local.f}
\end{figure}

\begin{table}[!t]
    \centering
    \begin{tabular}{c|c}
        \hline
        \textbf{Observation Region} & $Re f$ \\
        \hline
        $\Omega_{0}^{1}:=[0,0.5\pi]$     & 1.15346\% \\
        $\Omega_{0}^{2}:=[0,0.1\pi]$ & 2.00585\% \\
        $\Omega_{0}^{3}:=[0.3\pi,0.5\pi]$  & 1.07779\% \\
        $\Omega_{0}^{4}:=[0,0.1\pi]\cup[0.9\pi,\pi]$   & 1.39351\% \\
        \hline
    \end{tabular}
    \caption{ The relative errors of the source term $f$ under different observation regions $\Omega_{0}^{k},k=1,2,3,4.$}
    \label{Tab.1D.Local.f}
\end{table}

\subsection{Two-dimensional example}
In this subsection, setting $T=1,a=1,b=1$, we will consider the two dimensional complex Ginzburg-Landau equation
\begin{equation*}
	\begin{cases}
  u_{t}-(a+ib)\Delta u= e^{-t} f(x,y), & (x,y,t)\in\Omega_T, \\

 u(x,y,t)=0, & (x,y,t)\in\partial\Omega\times(0,T], \\

 u(x,y,0)=0, & (x,y)\in\Omega ,
    \end{cases}
\end{equation*}
where $\Omega=[0,\pi]\times[0,\pi]$, $f(x,y)=\sin x \sin y$, and the exact solution $u$ of the above equation is
\[
u(x,y,t)=\frac{\sin x \sin y}{2(a+ib)-\gamma } \left[e^{-t}-e^{-2(a+ib) t} \right].
\]
Let $h^{\delta}$ denote the observed data with noise $\delta>0$, defined as:
$$
h^{\delta}:=h(x,y)(1+\delta \cdot randn(0,1)),
$$
where $randn(0,1)$ represents a standard normal distribution.

\begin{example}{Global observation and experimental results.}

In the whole domain $\Omega=[0,\pi]\times[0,\pi]$, we take $u(x,y,1)$ as the observation data, and its precise expression is
$$
h^{\delta}(x,y)=u^{\delta}(x,y,1)=\frac{\sin x \sin y}{2(1+i)-1} \left[e^{-1}-e^{-2(1+i)} \right](1+\delta \cdot randn(0,1)).
$$
We consider three different noise levels: $\delta_1 = 0$, $\delta_2 = 0.01$, and $\delta_3 = 0.1$, aiming to study the impact of different noises on the experimental results under two-dimensional global observation.
    
\end{example}

We take $N=N_{\mathrm{int}}+N_{\mathrm{sb}}+N_{\mathrm{tb}}+N_d=256+256\times4+256+256\times1=1792$ as the number of collocation points. The experimental results are given in Figures \ref{Fig.2D.Global.f.0},  \ref{Fig.2D.Global.f.0.01}, \ref{Fig.2D.Global.f.0.1} and Table \ref{Tab.2D.Global.f}. 
As shown in Table \ref{Tab.2D.Global.f}), analysis of the relative error data for the inverted source term $f$ reveals that elevated observation noise levels progressively elevate the relative error in $f$. These errors remain within acceptable bounds, with each noise-level increment inducing only marginal changes in error magnitude, which indicates that our method can effectively reconstruct the two-dimensional source function $f$ through a single instantaneous observation in the whole observation domain.

\begin{figure}[htbp]
	\centering
	\includegraphics[width=1.0\linewidth]{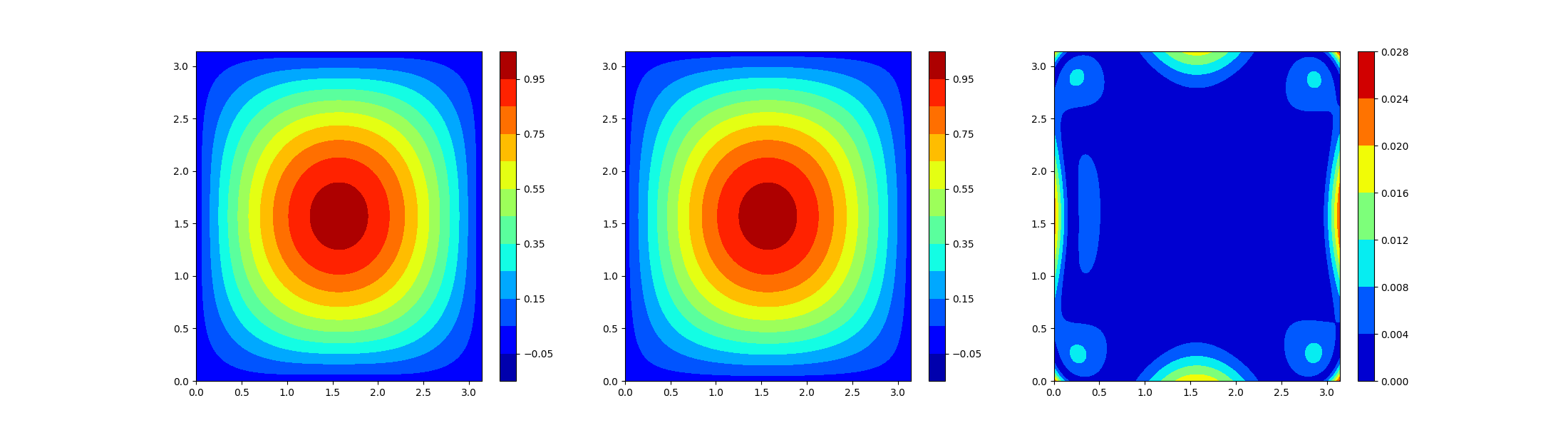}
	\caption{The numerical approximation (left), exact solution (middle), and relative error (right) of the source term $f$ under observation noise level $\delta_1=0$. }
	\label{Fig.2D.Global.f.0}
\end{figure}

\begin{figure}[htbp]
	\centering
	\includegraphics[width=1.0\linewidth]{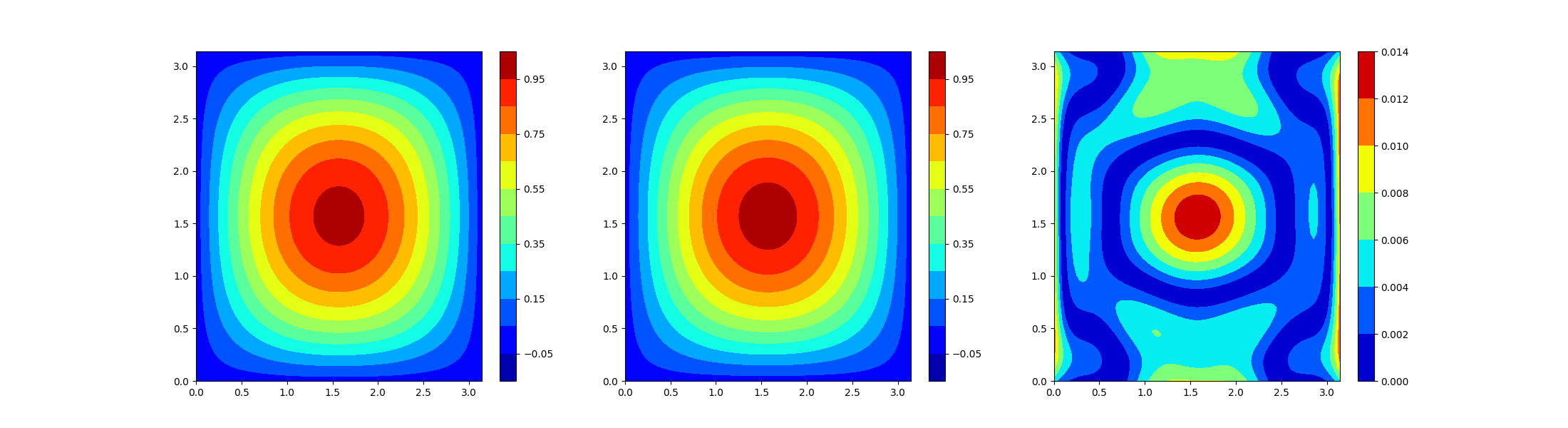}
	\caption{The numerical approximation (left), exact solution (middle), and relative error (right) of the source term $f$ under observation noise level $\delta_1=0.01$. }
	\label{Fig.2D.Global.f.0.01}
\end{figure}

\begin{figure}[htbp]
	\centering
	\includegraphics[width=1.0\linewidth]{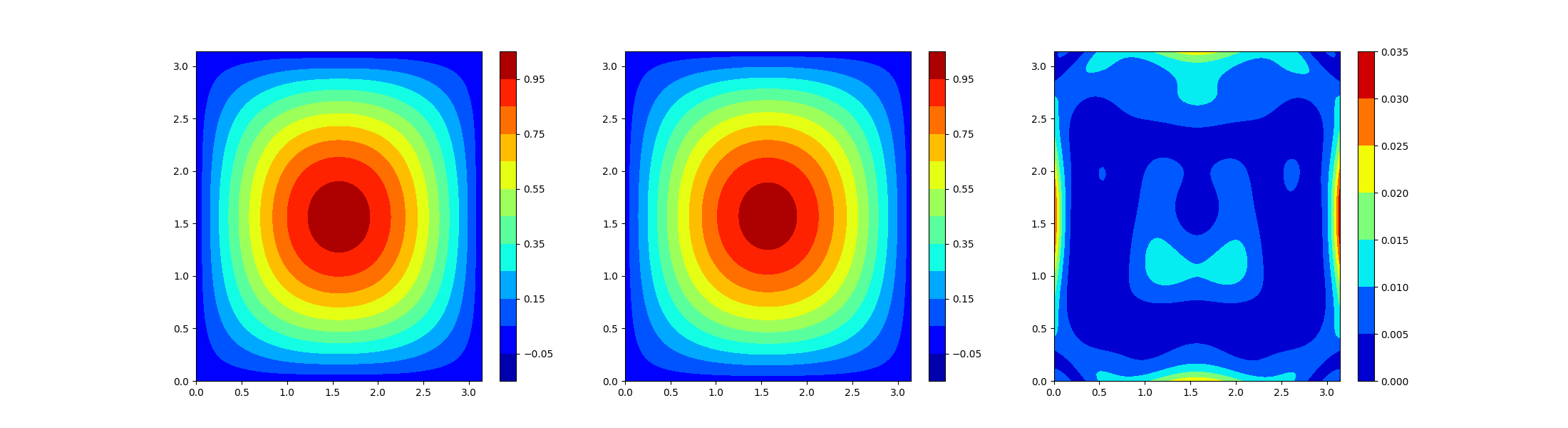}
	\caption{The numerical approximation (left), exact solution (middle), and relative error (right) of the source term $f$ under observation noise level $\delta_1=0.1$. }
	\label{Fig.2D.Global.f.0.1}
\end{figure}

\begin{table}[!t]
    \centering
    \begin{tabular}{c|c}
        \hline
        \textbf{Noise Level} & $Re f$ \\
        \hline
        $\delta_1=0$     & 0.92479\% \\
        $\delta_2=0.01$  & 1.02097\% \\
        $\delta_3=0.1$   & 1.52068\% \\
        \hline
    \end{tabular}
    \caption{ The relative errors of the source term $f$ under different observation noise levels $\delta_k,k=1,2,3$.}
    \label{Tab.2D.Global.f}
\end{table}

\begin{example}{Local observation and experimental results.}

In this example, we consider studying the inverse source term by taking subdomain observations at two instants. In the subdomain $\Omega_0\subset\Omega=[0,\pi]$, we take $u(x,y,1),u(x,y,0.5)$, $(x,y)\in\Omega_0$ as observation data, and their expressions with noise are
$$
h^{\delta}_1(x,y)=u^{\delta}(x,y,1)=\frac{\sin x \sin y}{2(1+i)- 1} \left[e^{-1}-e^{-2(1+i)}  \right](1+\delta \cdot randn(0,1)) ,
$$
and 
$$
h^{\delta}_2(x,y)=u^{\delta}(x,y,0.5)=\frac{\sin x\sin y}{2(1+i)- 1} \left[e^{-0.5}-e^{1+i} \right](1+\delta \cdot randn(0,1)) ,
$$
 where the noise level $\delta=0.001$. Consider the following local domains: $\Omega_{0}^{1}=[0,0.5\pi]\times[0,0.5\pi]$, $\Omega_{0}^{2}=[0,0.2\pi]\times[0,0.2\pi]$, $\Omega_{0}^{3}=[0.3\pi,0.5\pi]\times[0.3\pi,0.5\pi]$, and $\Omega_{0}^{4}=[0,0.1\pi]\times[0,0.1\pi]\cup [0.9\pi,\pi]\times[0.9\pi,\pi]$.
    
\end{example}

For the implementation details of the two-dimensional local data measurement, everything is consistent with the global data measurement except for the number of collocation points. In subdomains, we take $N=N_{\mathrm{int}}+N_{\mathrm{sb}}+N_{\mathrm{tb}}+N_d=256+256\times4+256+256\times2=2048$ collocation points.

The results after training are shown in Figures \ref{Fig.2D.Local.f.0.5}, \ref{Fig.2D.Local.f.0.2}, \ref{Fig.2D.Local.f.0.3-0.5}, \ref{Fig.2D.Local.f.0.1}, and Table \ref{Tab.2D.Local.f}. Upon analyzing the training results, it is evident that when the local observation domain is reduced, the relative error for the inverted source term $f$ increases. Moreover, with an identical size of the local observation domain, if it is closer to the boundary and further from the central region, the relative error of the source term $f$ also increases. For instance, in Figure \ref{Fig.2D.Local.f.0.1}, the relative error of $f$ is primarily concentrated in the central part. In summary, although the values of the relative error of the source term $f$ vary when we choose observation regions of different sizes and shapes, they are generally within an acceptable range. These results indicate that even if the observation domain is a subdomain, our method can effectively reconstruct the source function $f$ by observing two instants.

\begin{figure}[htbp]
	\centering
	\includegraphics[width=1.0\linewidth]{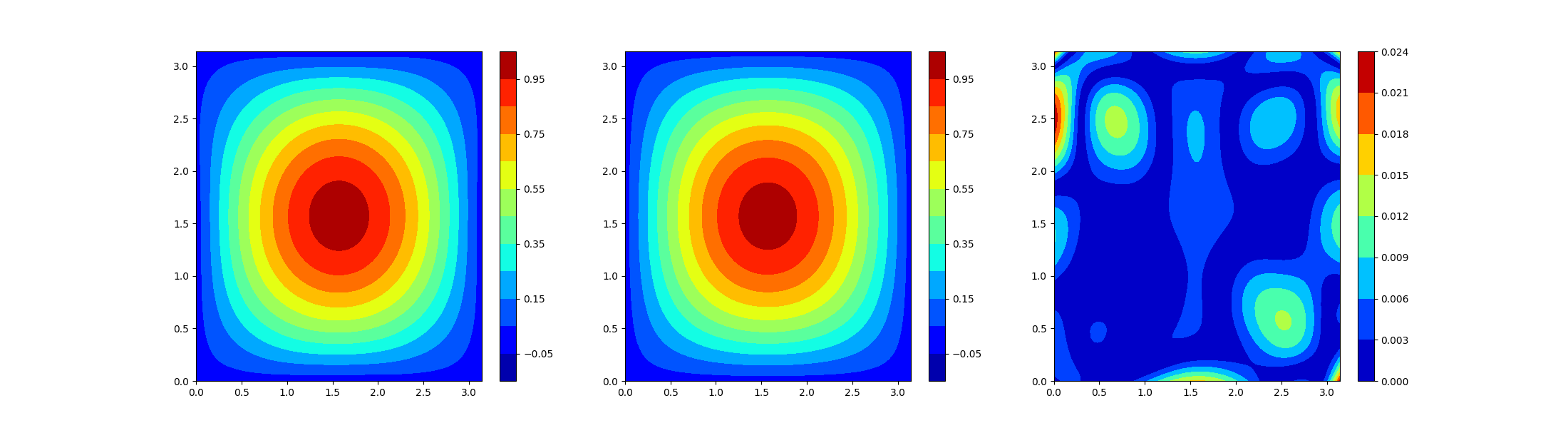}
	\caption{ The numerical approximation (left), exact solution (middle), and relative error (right) of the source term $f$ in the observation domain $\Omega_{0}^{1}$. 
 }
	\label{Fig.2D.Local.f.0.5}
\end{figure}

\begin{figure}[htbp]
	\centering
	\includegraphics[width=1.0\linewidth]{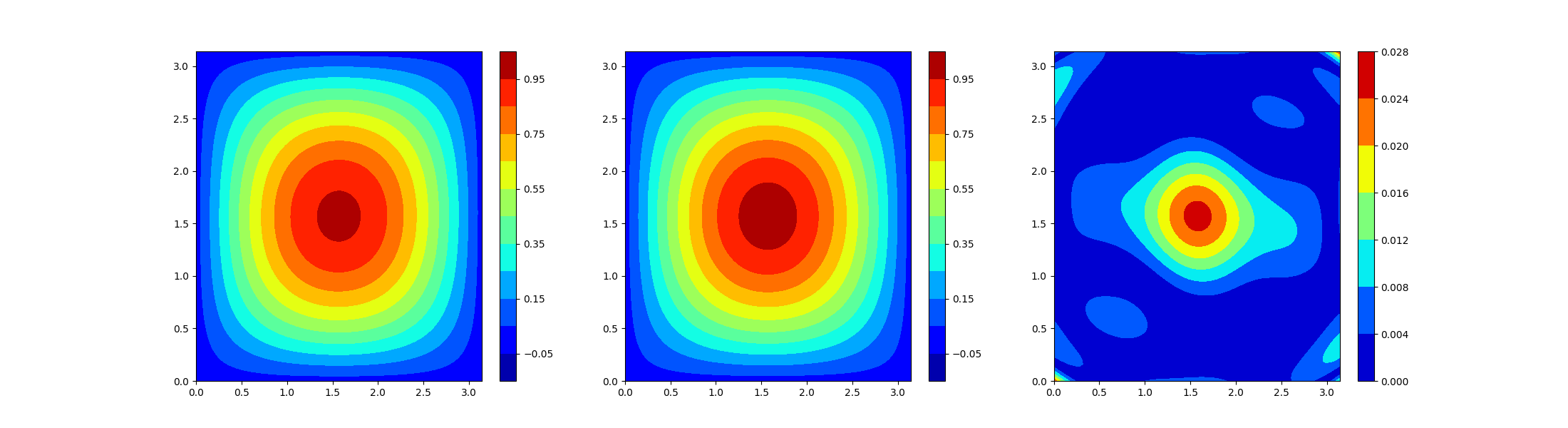}
	\caption{ The numerical approximation (left), exact solution (middle), and relative error (right) of the source term $f$ in the observation domain $\Omega_{0}^{2}$.}
	\label{Fig.2D.Local.f.0.2}
\end{figure}

\begin{figure}[htbp]
	\centering
	\includegraphics[width=1.0\linewidth]{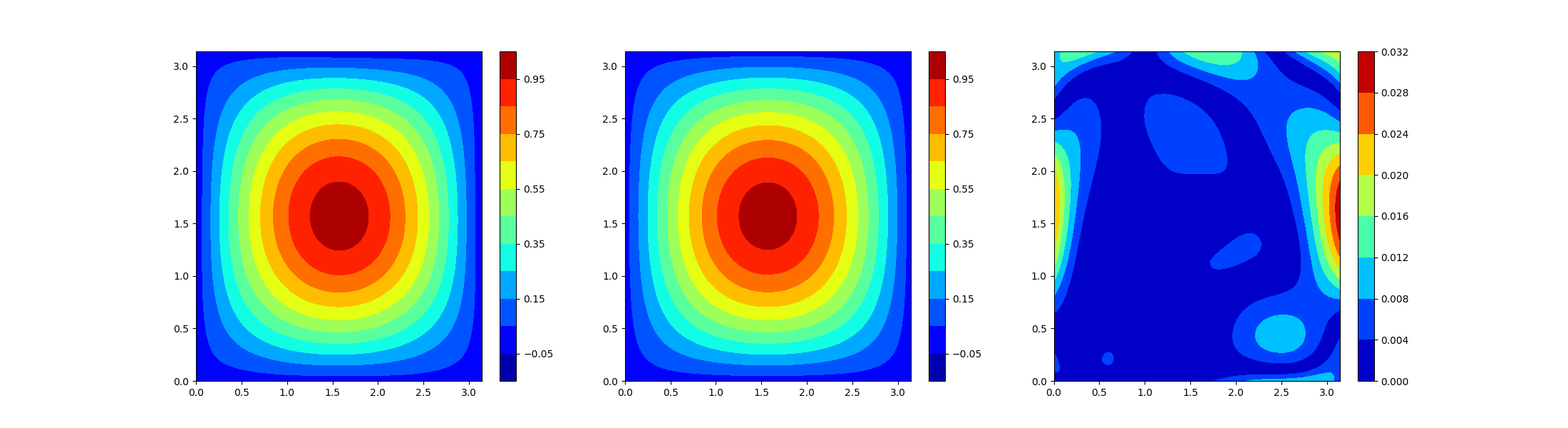}
	\caption{ The numerical approximation (left), exact solution (middle), and relative error (right) of the source term $f$ in the observation domain $\Omega_{0}^{3}$.}
	\label{Fig.2D.Local.f.0.3-0.5}
\end{figure}

\begin{figure}[htbp]
	\centering
	\includegraphics[width=1.0\linewidth]{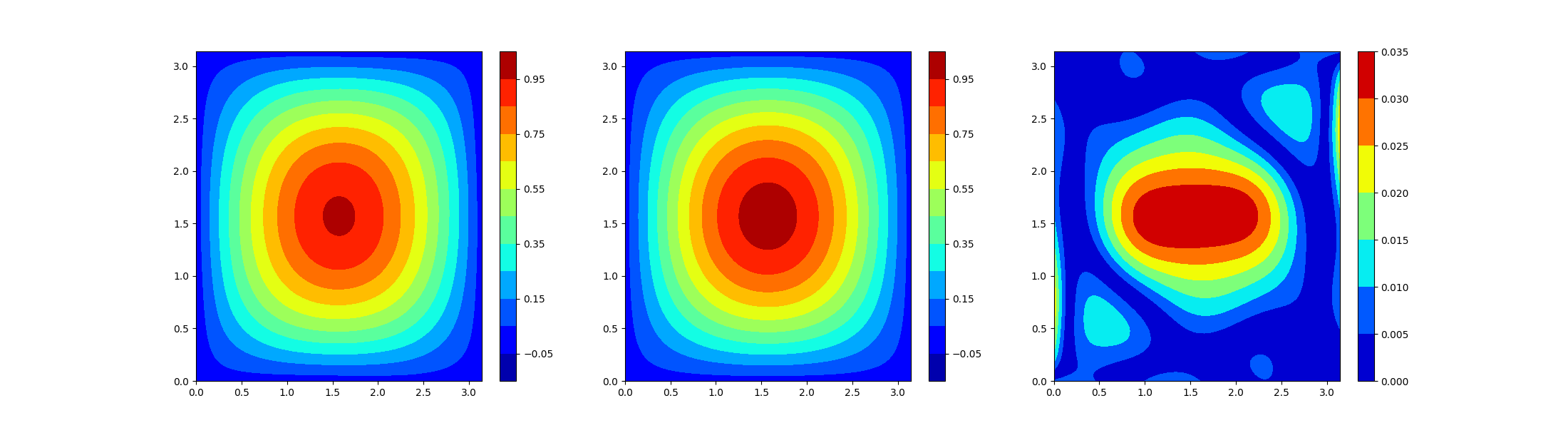}
	\caption{ The numerical approximation (left), exact solution (middle), and relative error (right) of the source term $f$ in the observation domain $\Omega_{0}^{4}$.}
	\label{Fig.2D.Local.f.0.1}
\end{figure}

\begin{table}[!t]
    \centering
    \begin{tabular}{c|c}
        \hline
        \textbf{Observation Region} & $Re f$ \\
        \hline
        $\Omega_{0}^{1}=[0,0.5\pi]\times[0,0.5\pi]$     & 1.02185\% \\
        $\Omega_{0}^{2}=[0,0.2\pi]\times[0,0.2\pi]$  & 1.38477\% \\
        $\Omega_{0}^{3}=[0.3\pi,0.5\pi]\times[0.3\pi,0.5\pi]$   & 1.35311\% \\
         $\Omega_{0}^{4}=[0,0.1\pi]\times[0,0.1\pi]\cup [0.9\pi,\pi]\times[0.9\pi,\pi]$   & 2.83682\% \\
        \hline
    \end{tabular}
    \caption{  The relative errors of the source term $f$ in different observation regions $\Omega_{0}^{k},k=1,2,3,4.$}
    \label{Tab.2D.Local.f}
\end{table}

\section{Concluding remarks}\label{sec-rem}
In this paper, under the assumption that \( g(t) = e^{-\gamma t} \) for a specified constant \(\gamma \in \mathbb C\), we investigate the inverse problem of recovering the source function \( f(x) \) in the initial-boundary value problem of \eqref{eq-gov}. We consider two approaches: one using the global data at the terminal time \( T \), and the other using measurements at two different instants \( 0 < T_1 < T_2 \). Specifically, the measurements at two different instants are given by
\[
u(x, T_j), \quad x \in \Omega_0 \subset \Omega, \quad j = 1, 2.
\]
For the first case, by employing the eigenfunction expansion method, we proved the uniqueness and conditional stability of the source identification under this observation setting.  
For the second case, the key to the proof is constructing a real analytic function from the combination of the measurements at these two instants. The local information of this analytic function can determine its global information, thereby providing the complete information of \( f(x) \). This approach allows us to prove the conditional stability of recovering \( f(x) \) by constructing the dual system and a variational norm stated in Lemma \ref{lem-dualnorm}.  

In terms of numerical implementation, we propose a novel loss function that incorporates derivative information of the PDE residuals as a regularization term. Based on the conditional stability of the inverse problem for the complex Ginzburg–Landau equation, we derive corresponding generalization error estimates. To validate the effectiveness of the proposed method, extensive numerical experiments are conducted in both one-dimensional and two-dimensional settings. The results demonstrate that the algorithm exhibits excellent performance in addressing the ill-posedness of the inverse problem, significantly enhancing the stability and accuracy of the solution.

Nevertheless, the method still faces certain challenges in deep neural network implementations, particularly when the loss function includes derivative terms of the residuals. In such cases, the need for long-time integration of the governing equations substantially increases the computational cost. To alleviate this issue, adaptive activation strategies \cite{jagtap2020,jagtap2022} can be employed to accelerate the convergence process. In future work, we aim to further investigate high-dimensional nonlinear inverse problems associated with the complex Ginzburg–Landau equation, with the goal of expanding the theoretical framework and exploring practical applicability, thereby advancing the development of data-driven inversion methodologies in this domain.

\section*{Declarations}
On behalf of all authors, the corresponding author states that there is no conflict of interest. No datasets were generated or analyzed during the current study.

\section*{Acknowledgments}
Z. Li thanks the National Natural Science Foundation of China (no. 12271277) and the Ningbo Youth Leading Talent Project (no. 2024QL045). M. Zhang is supported by National Natural Science Foundation of China (Grant No. 12301537), Natural Science Foundation of Hebei Province (Grant No. A2023202024) and Foundation of Tianjin Education Commission Research Program (Grant No. 2022KJ102). This work is partly supported by the Open Research Fund of the Key Laboratory of Nonlinear Analysis \& Applications (Central China Normal University), Ministry of Education, China and also the NSF of Jiangsu Province (Grant No. BK20221497 and No. BK20230036).


\begin{thebibliography}{10}
\bibitem{aranson2002}
I. S. Aranson and L. Kramer, \emph{The world of the complex Ginzburg–Landau equation}, Rev. Modern Phys. {\bf 74} (2002), no. 1, 99-143.


\bibitem{balanzario2020}
E. P. Balanzario and E. I. Kaikina, \emph{Regularity analysis for stochastic complex Landau-Ginzburg equation with Dirichlet white-noise boundary conditions}, SIAM J. Math. Anal. {\bf 52} (2020), no. 4, 3376–3396.


\bibitem{CL}
J. Cheng and J. Liu, \emph{An inverse source problem for parabolic equations with local measurements}, Appl. Math. Lett. {\bf 103} (2020), 106213, 7 pp.


\bibitem{CGZ0}
X. Cheng, C.-Y. Guo, and Y. Zheng, \emph{Global weak solution of 3-D focusing energy-critical nonlinear Schr\"odinger equation}, arXiv:2308.01226. 


\bibitem{cheng2023}
X. Cheng, C. Guo, and Y. Zheng, \emph{The limit theory of the energy-critical complex Ginzburg-Landau equation}, Acta Math. Sin. (Engl. Ser.), 2025.


\bibitem{correa2018}
W. J. Corrêa and T. \"Ozsari, \emph{Complex Ginzburg–Landau equations with dynamic boundary conditions}, Nonlinear Anal. Real World Appl. {\bf 41} (2018), 607-641.


\bibitem{dou2023}
F. Dou, X. Fu, Z. Liao, and X. Zhu, \emph{Global Carleman estimate and state observation problem for Ginzburg–Landau equation}, SIAM J. Control Optim. {\bf 61} (2023), no. 5, 2981–2996.


\bibitem{evans2010partial}
L.C. Evans, \emph{Partial differential equations}, Grad. Stud. Math., {\bf 19}. American Mathematical Society, Providence, RI, 2010, xxii+749 pp. ISBN: 978-0-8218-4974-3.


\bibitem{fan2005}
J. Fan and S. Jiang, \emph{Well-posedness of an inverse problem of a time-dependent Ginzburg–Landau model for superconductivity}, Commun. Math. Sci. {\bf 3} (2005), no. 3, 393-401.


\bibitem{fan2010}
J. Fan, S. Jiang, and G. Nakamura, \emph{Inverse problem of a time-dependent Ginzburg–Landau model for superconductivity with the final overdetermination}, Osaka J. Math. {\bf 47} (2010), no. 1, 89-108.


\bibitem{garciam2012}
V. Garc\'ia-Morales and K. Krischer, \emph{The complex Ginzburg–Landau equation: an introduction}, Contemp. Phys. {\bf 53} (2012), pp. 79–95.



\bibitem{ginzburg1950}
V. L. Ginzburg and L. D. Landau, \emph{On the theory of  superconductivity}, Zh. Eksp. Teor. Fiz. {\bf 20 }(1950), pp. 1064-1082.

\bibitem{hohenberg2015}
P. C. Hohenberg and A. P. Krekhov, \emph{An introduction to the Ginzburg-Landau theory of phase transitions and nonequilibrium patterns}, Phys. Rep. {\bf 572} (2015), 1-42.

\bibitem{kaikina2019}
E. Kaikina, \emph{Stochastic Landau-Ginzburg equation with white-noise boundary conditions of Robin type}, Nonlinearity {\bf 32} (2019), no. 12, 4967–4995.


\bibitem{kirane2017}
M. Kirane, E. Nane, and N. H. Tuan, 
\emph{On a backward problem for multidimensional Ginzburg–Landau equation with random data}, Inverse Problems {\bf 34} (2018), no. 1, 015008, 21 pp.


\bibitem{LiYa06}
G. Li and M. Yamamoto, \emph{Stability analysis for determining a source term in a 1$-D$ advection-dispersion equation}, J. Inverse Ill-Posed Probl. {\bf 14} (2006), no. 2, 147-155.

\bibitem{Liu2017}
W. Liu, W. Yu, C. Yang, M. Liu, Y. Zhang, and M. Lei, \emph{Analytic solutions for the generalized complex Ginzburg–Landau equation in fiber lasers}, Nonlinear Dynam. {\bf 89} (2017), no. 4, 2933-2939.

\bibitem{rosenstein2010}
B. Rosenstein and D. Li, \emph{Ginzburg-Landau theory of type II superconductors in magnetic field}, Rev. Modern Phys. {\bf 82} (2010), no. 1, 109-168.


\bibitem{rosier2008}
L. Rosier and B.-Y. Zhang, \emph{Controllability of the Ginzburg–Landau equation}, C. R. Math. Acad. Sci. Paris {\bf 346} (2008), no. 3-4, 167–172.


\bibitem{alberto2013super}
A. Salvio, \emph{Superconductivity, Superfluidity and holography},  J. Phys.: Conf. Ser. {\bf 442} (2013), 012040.



\bibitem{shimotsuma2016}
D. Shimotsuma, T. Yokota, and K. Yoshii, \emph{Existence and decay estimates of solutions to complex Ginzburg–Landau type equations}, 
J. Differential Equations {\bf 260} (2016), no. 3, 3119-3149.


\bibitem{trong2016}
D. D. Trong, B. T. Duy, and N. D. Minh, \emph{The backward problem for Ginzburg–Landau-type equation}, Acta Math. Vietnam. {\bf 41} (2016), no. 1, 143-169.



\bibitem{Zhang2023}
M. Zhang, Q. Li, and J. Liu, \emph{On stability and regularization for data-driven solution of parabolic inverse source problems}, J. Comput. Phys. {\bf 474 } (2023), Paper No. 111769, 20 pp.

\bibitem{Yamamoto2009}  
M. Yamamoto, \emph{Carleman estimates for parabolic equations and applications}, Inverse Problems {\bf 25} (2009), 123013.  

\bibitem{Bellassoued2010}  
M. Bellassoued and M. Yamamoto, \emph{Carleman estimates and an inverse heat source problem for the thermoelasticity system}, Inverse Problems {\bf 27} (2010), 015006.  

\bibitem{Wu2020}  
B. Wu, Q. Chen, and Z. Wang, \emph{Carleman estimates for a stochastic degenerate parabolic equation and applications to null controllability and an inverse random source problem}, Inverse Problems {\bf 36} (2020), 075014.  

\bibitem{Chen2022}
S. Chen, D. Jiang, and H. Wang, \emph{Simultaneous identification of initial value and source strength in a transmission problem for a parabolic equation}, Adv. Comput. Math. \textbf{48} (2022), no. 6, Article 77.

\bibitem{Zheng2014}
G. H. Zheng and T. Wei, \emph{Recovering the source and initial value simultaneously in a parabolic equation}, Inverse Problems \textbf{30} (2014), no. 6, 065013.

\bibitem{Wang2023}
Z. Wang, W. Zhang, and Z. Zhang, \emph{A data-driven model reduction method for parabolic inverse source problems and its convergence analysis}, J. Comput. Phys. \textbf{487} (2023), Paper No. 112156, 20 pp.

\bibitem{Liu2013}
J. C. Liu and T. Wei, \emph{A quasi-reversibility regularization method for an inverse heat conduction problem without initial data}, Appl. Math. Comput. \textbf{219} (2013), no. 23, 10866--10881.

\bibitem{Chen2020}
D. H. Chen, D. Jiang, and J. Zou, \emph{Convergence rates of Tikhonov regularizations for elliptic and parabolic inverse radiativity problems}, Inverse Problems \textbf{36} (2020), no. 7, 075001.

\bibitem{Le2013}
T. M. Le, Q. H. Pham, T. D. Dang, and others, \emph{A backward parabolic equation with a time-dependent coefficient: regularization and error estimates}, J. Comput. Appl. Math. \textbf{237} (2013), no. 1, 432--441.

\bibitem{Wen2025}
J. Wen, Y. L. Liu, X. J. Ren, \emph{et al.}, \emph{A regularization method for backward problems of singularly perturbed parabolic and fractional diffusion equations}, J. Appl. Anal. Comput. \textbf{15} (2025), no. 4, 2212--2237.

\bibitem{jagtap2020}
A. D. Jagtap, K. Kawaguchi, and G. E. Karniadakis, \emph{Adaptive activation functions accelerate convergence in deep and physics-informed neural networks}, J. Comput. Phys. {\bf 404} (2020), 109136.

\bibitem{jagtap2022}
A. D. Jagtap, Y. Shin, K. Kawaguchi, and G. E. Karniadakis, \emph{Deep Kronecker neural networks: A general framework for neural networks with adaptive activation functions}, Neurocomputing {\bf 468} (2022), 165–180.

\end{thebibliography}
\end{document}